\documentclass[11 pt]{amsart}
\usepackage[utf8]{inputenc}
\usepackage[english]{babel}
\usepackage{pifont}
\usepackage{verbatim}
\usepackage{moresize}
\usepackage{textalpha}
\usepackage{xcolor}
\usepackage[tmargin=1in,bmargin=1in,lmargin=1in,rmargin=1in]{geometry}
\usepackage[toc,page]{appendix}
\usepackage{lmodern}
\usepackage{amsmath}
\usepackage{amssymb}
\usepackage{mathtools}
\usepackage{proof}
\usepackage{enumitem}
\usepackage{physics}
\usepackage{bigfoot}
\usepackage{tikz-cd}
\usepackage{tikz}
\usetikzlibrary{fadings}
\usetikzlibrary{shadings}
\usepackage{ifthen}
\usepackage{bm}
\usepackage{verbatim}
\usepackage{graphicx}
\usepackage{subcaption}
\usepackage{amsthm}
\usepackage{braket}
\usepackage{makecell, multirow}
\usepackage{subfiles}
\usepackage{faktor}
\usepackage{xfrac}
\usepackage{simpler-wick}
\usepackage{imakeidx}
\usepackage{hyperref}
\usepackage{ulem}
\usepackage{pdflscape}
\usepackage{afterpage}
\usepackage{upgreek}
\makeindex[columns=3, title=Alphabetical Index, intoc]

\pagecolor{white}

\theoremstyle{plain}
\newtheorem{theorem}{Theorem}[section]

  \newtheorem{innercustomprop}{Proposition}
\newenvironment{myprop}[1]
  {\renewcommand\theinnercustomprop{#1}\innercustomprop}
  {\endinnercustomprop}

    \newtheorem{innercustomlem}{Lemma}
\newenvironment{mylem}[1]
  {\renewcommand\theinnercustomlem{#1}\innercustomlem}
  {\endinnercustomlem}

\newtheorem{lemma}[theorem]{Lemma}
\newtheorem*{claim}{Claim}

\newtheorem{corollary}[theorem]{Corollary}

\newtheorem{proposition}[theorem]{Proposition}

\theoremstyle{definition}
\newtheorem{definition}[theorem]{Definition}

\theoremstyle{remark}
\newtheorem{remark}[theorem]{Remark}

\newcommand{\C}{\mathbb{C}} 
\newcommand{\Hp}{\mathbb{H}} 
 
\newcommand{\D}{\mathbb{D}} 
 
\newcommand{\R}{\mathbb{R}}

\newcommand{\Z}{\mathbb{Z}} 
 
\newcommand{\N}{\mathbb{N}} 

\newcommand{\reg}{\mathrm{reg}}

\newcommand{\vertiii}[1]{{\left\vert\kern-0.25ex\left\vert\kern-0.25ex\left\vert #1 
\right\vert\kern-0.25ex\right\vert\kern-0.25ex\right\vert}}

\newcommand{\mcalA}{\mathcal{A}} 
\newcommand{\mcalB}{\mathcal{B}}

\newcommand{\mcalN}{\mathcal{N}}

\newcommand{\mcalX}{\mathcal{X}}

\newcommand{\ad}[1]{\di A( #1 )}
\newcommand{\ld}[1]{\abs{\di #1 }}
\newcommand{\adz}{\ad{z}}

\newcommand{\ldx}{\ld{x}}
\newcommand{\ldy}{\ld{y}}

\newcommand{\mbbP}{\mathbb{P}}
\newcommand{\mbbQ}{\mathbb{Q}}
\newcommand{\bfb}{\mathbf{b}}

\newcommand{\bin}{\mcalB_{\text{in}}}
\newcommand{\sbin}{b_{\text{in}}}
\newcommand{\bout}{\mcalB_{\text{out}}}

\newcommand{\inset}[1]{\left \{ #1 \right \}}
\newcommand{\paren}[1]{\left ( #1 \right )} 

\renewcommand{\det}{\mathrm{det}}

\DeclareMathOperator{\cle}{CLE} 

\newcommand{\Pb}[1]{\mathbb{P}\left[#1 \right]} 

\newcommand{\Ef}[2]{\mathbb{E}^{#1} \left[#2 \right]}
\newcommand{\Eb}[2]{\mathbb{E}_{#1} \left[#2 \right]}

\newcommand{\ind}[1]{\mathbf{1}_{#1}}
\newcommand{\ann}[1]{\mathbb{A}_{#1}}

\newcommand{\E}[1]{\mathbb{E}\left[#1 \right]}

\newcommand{\sle}{\mathrm{SLE}}
\newcommand{\slek}{\mathrm{SLE}_\kappa}
\newcommand{\dist}{\mathrm{dist}} 
\newcommand{\di}{\mathrm{d}}

\newcommand{\scal}[1]{\left\langle #1 \right\rangle} 

\newcommand{\normh}[1]{\norm{#1}_{\nabla(D)}}
\newcommand{\normreg}[1]{\norm{#1}_{\nabla(D), \mathrm{reg}}}

\newcommand{\restr}[1]{|_{#1}} 
\newcommand{\inv}[1]{\frac{1}{#1}} 
\newcommand{\sinv}[1]{\sfrac{1}{#1}}

\newcommand{\al}{\alpha} 
 
\newcommand{\ga}{\gamma} 
\newcommand{\de}{\delta} 
\newcommand{\eps}{\varepsilon} 
 
\newcommand{\ka}{\kappa} 
\newcommand{\lam}{\lambda}

\renewcommand{\phi}{\varphi}

\newcommand{\cc}{\mathbf{c}}

\setcellgapes{2pt}
\makegapedcells

\allowdisplaybreaks

\title[SLE in multiply connected domains]{$\mathbf{\slek}$ and its partition function in multiply connected domains via the Gaussian Free Field and restriction measures}
\author{JUHAN ARU, PHILÉMON BORDEREAU}
\date{}

\begin{document}

\begin{abstract}
    One way to uniquely define Schramm-Loewner Evolution (SLE) in multiply connected domains is to use the conformal restriction property. This gives an implicit definition of a $\sigma$-finite measure on curves; yet it is in general not clear how to construct such measures nor whether the mass of these measures, called the partition function, is finite.
    
    We provide an explicit construction of such 
    SLEs in multiply connected domains when $\kappa = 4$ using the Gaussian Free Field (GFF). In particular, we show that there is a mixture of laws of level lines of GFFs that satisfies the restriction property, both when the target points of the curve are on the same or on distinct boundary components. This allows us to give an expression for the partition function of $\sle_4$ on multiply connected domains and shows that the partition function is finite, answering a question raised in \cite{lawler_partition_2009,lawler_defining_2011}. 

    In a second part, we provide a second construction of $\slek$ in multiply-connected domains for the whole range $\kappa \in (8/3,4]$: specific, however, to the case of the two target points belonging to the same boundary components. This is inspired by \cite{werner_clekappa_2013} and consists of a mixture of laws on curves obtained by following $\cle_\ka$ loops and restriction hulls attached to parts of the boundary of the domain. In this case as well, we obtain as a corollary the finiteness of the partition function for this type of $\slek$.
\end{abstract}

\maketitle

\section*{Introduction}

Schramm-Loewner Evolution $\slek$, $\kappa \geq 0$ is a one-dimensional family of conformally invariant probability measures on non-crossing curves in simply connected domains of the complex plane \cite{schramm_scaling_2000} that are either proved or conjectured to describe the scaling limit of interfaces of a variety of planar statistical physics models, e.g. \cite{smirnov_conformal_2010, lawlerLARWUST2004, schramm_contour_2009,duminil-copin_conformal_2021, chelkak_convergence_2014, kenyon_boundary_2010, sheffield_exploration_2009}. It also arises in connection with Liouville quantum gravity---see e.g. \cite{gwynne_mating_2023} for a review---and, more relevant to the present work, it is expected to be linked with Conformal Field Theories (CFTs) of central charge $\cc=\cc(\kappa)=(6-\kappa)(3\kappa-8)/(2\kappa)$, see e.g. \cite{peltola_towards_2019} for a review and \cite{baverez_cft_2024, gordina_infinitesimal_2025} for recent advances in this direction.

In his seminal work, Schramm constructed (chordal) $\slek$ in simply connected domains as a solution to Loewner's equation driven by a multiple of Brownian motion $\sqrt{\kappa} B$ and showed via a concise argument that the resulting law is the unique law on simple curves (modulo increasing time-reparametrisation) that satisfies conformal invariance and a domain Markov property \cite{schramm_scaling_2000}. However, neither the construction via Loewner's equation nor the characterisation generalise directly to multiply connected domains as the conformal class changes with the growth of the curve and therefore the remaining slit domain cannot be related to the initial one via conformal transformations. In other words, due to the non-triviality of the moduli space for more general domains, identifying a family of laws on curves indexed by the moduli space requires specifying how these laws between different conformal classes compare. 

On the other hand, variants of $\slek$ have been constructed in doubly connected domains via the annulus Loewner equation using the modulus of the slit domain as the time parametrisation \cite{zhan_stochastic_2004, zhan_reversibility_2008, hagendorf_gaussian_2010, hagendorf_sle_2008}, and in higher connectivity using further generalisations of the Loewner equation \cite{bauer_chordal_2008, bauer_stochastic_2004}. In these latter works, one also needs to keep track of the evolution of the moduli as the curve grows in order to obtain a Markov property: the downside is that the natural conditions do not allow to single out a unique law on the curves, yielding instead a whole family of possible curves as $\slek$. A different approach that does single out a unique law of $\slek$ was suggested by Lawler \cite{lawler_partition_2009} and is described next.

\subsection*{Lawler's definition of \texorpdfstring{$\sle$}{SLE} in multiply connected domains via conformal restriction}

The (conformal) restriction property describes how the law of SLE in a domain $D$, on the event that it remains in a subdomain $D'$, compares to the law of SLE in $D'$.
More precisely, in the simply connected setting, consider $K \subset \overline{\Hp}$ compact with $\dist(K,0)>0$ and $\Hp \setminus K$ (connected and) simply connected. It was shown in \cite{lawler_conformal_2003} that for $\kappa \leq 4$ (the values of $\ka$ for which the resulting curve is a.s. simple \cite{rohde_basic_2005}) the law $\nu_{\Hp \setminus K, \kappa}(0,+\infty)$ of $\slek$ from $0$ to $\infty$ is absolutely continuous with respect to $\nu_{\Hp, \kappa}(0, +\infty)$ and that their Radon-Nikodym derivative is explicitly given by
\begin{equation}\label{eq:restrictionintroHp}
\frac{\di \nu_{\Hp\setminus K, \kappa}(0,+\infty)}{\di \nu_{\Hp, \kappa}(0,+\infty)}(\eta) = \ind{\eta \subset \Hp \setminus K} \abs{\varphi_K'(0)}^{-h} \exp(\frac{\cc}{2} m_{\Hp}(\eta, K)),
\end{equation}
where $\varphi_K:\Hp \setminus K \to \Hp$ is the conformal equivalence chosen such that $\varphi_K(0)=0$ and $\varphi_K(z) \sim z$ as $z \to \infty$, and $h=h_\ka=(6-\kappa)/(2\kappa)$ is a constant that can be related to the weight of a boundary condition changing operator in a CFT \cite{cardy_sle_2005, peltola_towards_2019, dubedat_sle_2015}. The term $m_{\Hp}(\eta, K)$ corresponds to the Brownian measure (see \cite{BrownianLoopSoup} for a definition thereof) of loops in $\Hp$ that intersect $K$ and $\eta$. This restriction property is one of the motivations for seeing the laws $\nu_{D, \kappa}(x,y)$ no longer as probability measures, but rather as measures of varying total mass. For $D$ simply connected and $\partial D$ sufficiently regular at $x$ and $y$ this total mass is defined as \cite{lawler_partition_2009, lawler_defining_2011}
\begin{equation*}
Z_{D, \kappa}(x,y)=
\begin{cases}
    H_{\partial D}(x,y)^h & \text{if } x, y \in D \setminus \inset{\infty}, x \neq y \\
    \abs{\phi_K'(x)}^h & \text{if } y=\infty
\end{cases}
\end{equation*}
where $H_{\partial D}$ denotes the boundary Poisson kernel  and $\phi_K:D \to \Hp$ is the conformal equivalence with $\phi_K(x)=0$ and $\phi_K(z)\sim z$ as $z \to \infty$, so that the term $\abs{\phi_K'(0)}^{-h}$ in Equation~\eqref{eq:restrictionintroHp} disappears. This total mass is called the partition function of the chordal $\slek$, in general defined up to an arbitrary multiplicative constant, but here fixed by the definition above since $Z_{\Hp, \kappa}(0,+\infty)=1$ for all $\kappa >0$. The measures $\nu$ then satisfy a form of conformal covariance
\begin{equation}\label{eq:confcovintro}
\varphi_\ast \nu_{D, \kappa}(x,y) = \abs{\varphi'(x)}^h \abs{\varphi'(y)}^h \nu_{\varphi(D), \kappa}(\varphi(x), \varphi(y))
\end{equation}
for distinct $x, y \in \partial D$ and $\varphi:D \to \varphi(D)$ a conformal equivalence (that can be extended to the boundary under some assumptions that we detail later on), suggesting that these partition functions may be also interpreted as operators in a Conformal Field Theory \cite{dubedat_sle_2015} satisfying BPZ equations.

Lawler suggested that the restriction property \eqref{eq:restrictionintroHp} could be used to implicitly but uniquely define $\sle$ also on multiply connected domains $D$ as follows (see Definition \ref{def:sle} for more precision):
\begin{definition}[Conformal restriction $\slek$ in multiply-connected domains]\label{def:sleintro}
For $0 < \kappa \leq 4$, \emph{chordal $\slek$} is the family of laws $\inset{\nu_{D, \kappa}(x,y)}_{(D,x,y)}$, where $D \subset \C$ is a finitely-connected domain with smooth boundary at distinct $x, y \in \partial D$, such that 
\begin{itemize} 
\item conformal covariance is satisfied as in Equation~\eqref{eq:confcovintro}
\item the restriction of $\nu_{D, \kappa}(x,y)$ to curves remaining in a simply connected domain $D' \subset D$ 
is absolutely continuous with respect to the $\slek$ measure $\nu_{D', \kappa}(x,y)$ in $D'$, with Radon-Nikodym derivative 
\begin{equation}\label{eq:restrictionintro}
\frac{\di \nu_{D, \kappa}(x,y)\restr{D'}}{\di \nu_{D', \kappa}(x,y)}(\eta) = \ind{\eta \subset D'} \exp(-\frac{\cc}{2} m_{D}(\eta, D \setminus D')).
\end{equation}
\end{itemize}
\end{definition}
Because these Radon-Nikodym derivatives are consistent, one can see (for instance via a Carathéodory extension argument) that Equation~\eqref{eq:restrictionintro} indeed implicitly defines a ($\sigma$-finite) measure on curves on any finitely connected domain of the complex plane, since the additivity structure of the Brownian loop measure makes the restriction to different domains compatible \cite{lawler_defining_2011}. Moreover, these measures are unique up to a multiplicative constant fixed by requiring that $Z_\Hp(0,+\infty)=1$, and satisfy the domain Markov property that is expected of SLE, as well as a form of reversibility. This however raises a natural question: are the partition functions, i.e. the total masses $Z_{D, \kappa}(x,y):= \abs{\nu_{D,\ka}(x,y)}$ positive and finite for general multiply-connected domains $D$? If this were the case, the measures could be normalised and the resulting probability measures would be conformally invariant.

Beyond the simply-connected case, it is shown in \cite{lawler_defining_2011} that $Z_{D, \kappa}(x,y)$ is indeed finite either when $\kappa \leq 8/3$ (corresponding to the regime $\cc \leq 0$) or when $\kappa\leq 4$ and $D$ is a conformal annulus. The main aim of this paper is to give explicit constructions of the laws satisfying the definition above, and to use those to prove that the partition function is finite for more general finitely connected domains. 

Our results are based on two other ways of constructing $\sle$ in simply-connected domains, which we will describe next. We will comment on further directions at the end of the introduction.

\subsection*{Two other constructions of \texorpdfstring{$\sle$}{SLE} in simply connected domains}
In the case of simply connected domains, $\sle$ curves can be seen to appear from other conformally invariant objects. In particular, the two following connections are instrumental for the present article:
\begin{itemize}
    \item \textit{As a level line of the Gaussian free field.}
    The Gaussian Free Field (GFF) in a proper domain of the complex plane corresponds to a centred Gaussian process with covariance operator formally given by the Green's function. Its connection with $\sle$ was discovered in the seminal paper \cite{schramm_contour_2009}, where it was proved that $\sle_4$ arises as a scaling limit of a certain level-line of the discrete GFF, and in \cite{schramm_contour_2013} (see also \cite{dubedat_sle_2009}) it was shown how to define the level line purely in the continuum. More precisely, a curve $\eta \sim \nu_{D, 4}(x,y)$ in a simply-connected domain $D$ with distinct $x, y \in \partial D$ can be coupled with a Gaussian Free Field $\Phi_D$ with mean corresponding to the harmonic extension of the boundary values $- \lambda:=-\sqrt{\pi/8}$ (resp. $+\lambda$) on the counterclockwise boundary arc $(yx)$ (resp. $(xy)$), so that conditionally on $\eta$, $\Phi_D$ is equal in law to a sum of two independent GFFs on the left (resp. on the right) of $D \setminus \eta$, with constant mean equal to $-\lam$ (resp. $+\lam$). Importantly, $\eta$ is then also a measurable function of $\Phi_D$. Further works explore variations of this results for other types of SLEs \cite{powell_level_2017, hagendorf_gaussian_2010, izyurov_hadamards_2013, lupu_level_2024}, and prove the existence of level-lines for GFFs in multiply connected domains \cite{aru_bounded-type_2019}. \footnote{In fact, for $\ka < 4$, $\slek$ can be also coupled to the GFF as a certain flow line of the GFF \cite{miller_imaginary_2016, sheffield_conformal_2016}, but for now these constructions have not been adapted to the multiply-connected setting and it is not clear whether some of our computations would generalise even if they were.}
    
    \item \textit{Via restriction measures and CLEs.} The (one-sided) restriction measures and the Conformal Loop Ensemble (CLE) are one-parameter families of measures on hulls (resp. loops) satisfying a form of conformal Markov property. Both can be constructed from Brownian motion: for $D \subset \C$ a bounded domain with distinct $x, y \in \partial D$, the filling (relatively to $(xy)$) of a Poisson point process of measure proportional (of factor, or intensity $\al$) to a measure of Brownian excursions on the boundary arc $(yx) \subset \partial D$ yields a one-sided restriction measure of parameter $h=h(\al)$ (see Section~\ref{subsec:restriction} and \cite{lawler_conformal_2003}), whereas the collection of outer boundaries of clusters of loops generated by a Poisson point process of intensity $c$ with respect to the Brownian loop measure corresponds to a $\cle_\ka$ with $\kappa=\kappa(c)$ (see Section~\ref{subsec:cle} and \cite{sheffieldConformalLoopEnsembles2012}). In \cite{werner_clekappa_2013}, it was shown that for $D \subset \C$ a simply connected domain with distinct $x, y \in \partial D$ and $\kappa \in (8/3,4]$, if one samples independently a restriction measure on $(yx)$ of parameter $h_\kappa=(6-\ka)/(2\ka)$ and a $\cle$ of parameter $\ka$, and then proceeds to follow the outer boundary of the restriction sample along with the $\cle$ loops it intersects, then the curve obtained is distributed like $\sle_\ka$ in $D$.
\end{itemize}

\subsection*{Summary of results} 

In this article, we extend these connections to the case of multiply connected domains in a way that provides an explicit construction of the $\slek$ described by Definition \ref{def:sleintro}. We sometimes refer to this measure as \textit{conformal restriction SLE} to emphasise that a priori different properties might be used to define $\slek$ in multiply connected domains and they might give rise to different measures. 

In the first part, we provide a construction specific to $\sle_4$ that relies on the connection with the GFF, both when the two target points (i.e. start and end-point of SLE) belong to the same boundary component (\emph{non-crossing case}) or different boundary components (\emph{crossing-case}). Informally, the results for $\sle_4$ can be stated as follows:
\begin{itemize}
\item[$\star$] The law of conformal restriction $\sle_4$ in any multiply connected domain can be obtained from a mixture of level-lines of the GFF via appropriate topological conditionings. (See Theorem~\ref{thm:noncrossing},~\ref{thm:crossing})
\item[$\star$] The partition function of conformal restriction $\sle_4$ can be identified, and is finite and positive. (See Corollary~\ref{cor:partitionfunctionnoncrossing4},~\ref{cor:partitionfunctioncrossing4})
\end{itemize}

Somewhat surprisingly, except in the crossing-case in the annulus where it can be interpreted as a level line of a compactified GFF, the conformal restriction $\sle_4$ cannot in general be seen as a level line of a single GFF, but is rather given as a mixture of conditioned level lines. For instance in the non-crossing case, a simple curve connecting distinct $x, y \in \partial D$ lying on the same boundary component partitions the other boundary components into $\Upsilon_l \cup \Upsilon_r$ where $\Upsilon_l$ is the set of those on its left and $\Upsilon_r$ the set of those on its right. The mixture of laws yielding $\sle_4$ can then be understood as a mixture of certain conditional laws: for each such partition, we sample a level-line for a GFF whose boundary value is $-\lam$ on $\Upsilon_l$, $+\lam$ on $\Upsilon_r$, and finally \textit{condition} this level-line to indeed partition the boundary components into exactly $\Upsilon_l \cup \Upsilon_r$. Such a mixture cannot be obtained by just choosing appropriate boundary conditions upfront, see Remark~\ref{rem:levelline}.

Except for topological discrepancies, the proofs for both cases then follow the same strategy: the law of the level-line $\eta$ of a GFF $\Phi_D$ in a domain $D$ with suitable boundary conditions conditioned to remain in a smaller domain $D' \subset D$ can be seen as a level-line for a GFF in $D'$ with random boundary conditions given by the values $\xi$ of $\Phi_D$ on $\partial D'$. We first compare this law to $\nu_{D',4}(x,y)$ for $\xi$ fixed in Propositions~\ref{prop:computderiv} and \ref{prop:computderivcrossing}. This builds on the techniques developed in \cite{aru_first_2020, aru_extremal_2022} to compare the laws of local sets of GFFs under changes of boundary conditions: Proposition~\ref{prop:computderiv} can be seen as a variant of Proposition 3.3 in \cite{aru_extremal_2022} for more general boundary conditions, but specific to SLEs.
In a second step, we compute the average of this Radon-Nikodym derivative over $\xi$, conditionally on $\eta \subset D'$. This is a Gaussian integration, and its result consists of a determinant term corresponding to a mass of loops by Proposition 2.2 in \cite{dubedat_sle_2009}, and (an exponential of) a Dirichlet energy which we identify as the partition function of the SLE. Equivalently, it can be seen as the inverse extremal distance between parts of the boundary of the domain. \\

In a second part, we provide another construction of $\sle_\ka$, for the whole range $\ka \in (8/3,4]$, but specific to the non-crossing case:
\begin{itemize}
\item[$\star$] In any multiply-connected domain and for any $\ka \in (8/3, 4]$ we can construct the conformal restriction $\slek$ with the endpoints on the same boundary components as a mixture of laws obtained from Brownian excursions and CLE via topological conditioning. (See Theorem~\ref{thm:noncrossingslekappa})
\item[$\star$] In the setup just described, one can give an expression for the partition function and show that it is finite. (See Corollary~\ref{cor:partitionfunctionnoncrossingkappa})
\end{itemize}
More precisely, conditionally on the boundary components on the left of the curve, one samples a Poisson point process of Brownian excursions on these boundary components plus the arc $(yx)$ and an independent $\cle$, and follows the right boundary together with the $\cle$ loops encountered along the way. To verify that the restriction property \eqref{eq:restrictionintro} is satisfied, it is sufficient to compute the Radon-Nikodym derivative of this law against that in a domain $D \setminus K$, $K$ a hull at positive distance from $(yx)$: the contribution from the $\cle$ provides the term $\exp(\cc \cdot m_D(\eta, K)/2)$, whereas the one for the Brownian excursions yields a normalising constant that we identify as the partition function of the $\sle$ (in the simply connected case, this term would be $H_{\partial D}(x,y)^h$). We were however not able to extend this construction to the crossing case, which appears to be more delicate - it is not clear how to use the geometric construction via Brownian excursions and CLE to define a crossing $\slek$, even in the simplest case of $\sle_4$.\\

\subsection*{Further constructions and connections}
We now comment on related work concerning SLEs in non-simply connected domains and their potential connections to lattice models.

First, parallel to Lawler's work \cite{lawler_defining_2011} showing finiteness of the partition function on annuli, Zhan \cite{zhan_stochastic_2004, zhan_reversibility_2008} introduces a family of processes $\sle(\kappa, \Lambda_{\kappa; \scal{s}})$ as the unique law on curves in annuli satisfying the domain Markov property and reversibility, and shows \cite{zhan_restriction_2012} that conformal restriction is also satisfied under a setting slightly different from the above. In fact, it is argued in \cite{lawler_defining_2011} that both constructions coincide when the laws are normalised. \cite{izyurov_hadamards_2013} also introduce a variety of SLE curves in the annulus as level-lines of GFFs with various boundary conditions and describe their driving function relative to the annulus Loewner equation, but the conformal restriction properties are not discussed.

Second, the restriction property also appears in the study of discrete models conjectured to converge to $\sle$ in general domains. An example of such a model is the $\lam-$Self Avoiding Walks (SAW), introduced in \cite{kozdron_configurational_2006}. Already in the discrete, it satisfies a certain random-walk related restriction property and this property should give rise to the the restriction property described above in the scaling limit. See \cite{lawler_defining_2011} for conjectures related to the $\lam-$SAW and $\slek$, which should also hold in the non-simply-connected setting. The case of the more standard lattice models is less clear---for example it not obvious which Ising interface, or with which boundary conditions, would converge either to the crossing or non-crossing annulus $\sle_3$ defined above. We find it is an interesting question to figure it out.

Finally, there is yet a third fruitful way via Liouville quantum gravity and welding of random surfaces \cite{sun_sle_2023} that we will not treat in this current paper (and that is a priori harder to use for the case $\sle_4$ that is the main topic here).

\subsection*{Structure of the article}
The article is structured as follows. Section~\ref{sec:preliminaries} reviews standard background on complex analysis tools and $\sle$. We move the reminders about the GFF, resp. the restriction measures and $\cle$ to the second and third sections where these tools appear specifically. In Section~\ref{sec:sle4}, we construct $\sle_4$ first in the non-crossing case then in the crossing case, from the GFF.
Section~\ref{sec:slek} provides another construction of the $\sle_\kappa$ measure for $\ka \in (8/3,4]$  via the Brownian excursions and $\cle$, but specific to the non-crossing setting.

\subsection*{Acknowledgements} Both authors are supported by Eccellenza grant 194648 of the Swiss National Science Foundation and are members of NCCR Swissmap. The authors thank Baptiste Cerclé and anonymous referees for helpful comments.

\section{Preliminaries}\label{sec:preliminaries}
\subsection{Notations}\label{subsec:notations}

We establish notation and definitions that will be used throughout this text:
\begin{itemize}
    \item $\D$ denotes the open unit disk in $\C$, $\Hp$ the open upper half plane, and for $p>0$, $\ann{p}$ denotes the annulus $\ann{p}:= \inset{z \in \C: e^{-p} < \abs{z} <1}$.
    A \emph{domain} $D$ is a non-empty proper open subset of $\C$.

    \item We call a simple curve in a domain $D$ an injective and continuous map $\eta:[0,t_\eta] \to D$ for some $0 < t_\eta \leq + \infty$, and write $\eta_t:= \eta([0,t])$ for $t \in [0,t_\eta]$. We in fact allow curves to have endpoints on $\partial D$, i.e. that $\eta(0),\eta(t_\eta) \in \partial D$, write $\eta \subset D$ if $\eta((0,t_\eta)) \subset D$, and alternatively call the trace of those curves crosscuts from $\eta(0)$ to $\eta(t_\eta)$ in $D$. When unspecified, simple curves (that we sometimes also call Jordan curves) are understood modulo increasing time-reparametrisation.
    If $\eta(0)$ and $\eta(t_\eta)$ belong to the same boundary component, it separates $D$ in two parts we call $D_l$ or $(D\setminus\eta)_{\text{left}}$ (resp. $D_r$ or $(D\setminus\eta)_{\text{right}}$) the left (resp. right) connected component of $D \setminus \eta$ relatively to the orientation of $\eta$.
    
    \item We call a \emph{(non-crossing) admissible domain} a triplet $(D,x,y)$ where $D \subset \D$ is a connected domain whose boundary $\partial D$ can be written $\mcalB \cup \bigcup_{j=1}^n \mcalB_j$, $n\geq 0$, with $\mcalB$ and each $\mcalB_j$ given by a closed Jordan curve, and $x,y \in \mcalB \cap \partial \D$. We further require that $\mcalB$ agrees with $\partial \D$ in a neighbourhood of $x$ and $y$. We denote by $(xy)$ or $\mcalB_r$ (resp. $(yx)$ or $\mcalB_l$) the counterclockwise arc from $x$ to $y$ (resp. from $y$ to $x$) in $\mcalB$ (see Figure~\ref{fig:definitions} for a sketch).

    On the other hand, we call a \emph{(crossing) admissible domain} the data $(D,x,y)$ where $D \subset \D$ is a finitely connected domain $D\subset \D$ with $\partial D = \bin \cup \bout \cup \bigcup_{j=1}^n \mcalB_j$, $n \geq 0$, with each set in the union corresponding to a closed Jordan curve, together with $x \in \partial \D$ and $y \in \D$. We also require that $x \in \bout$ with $\bout$ and $\partial \D$ agreeing locally around $x$, and that $y \in \bin$ is $\partial D$-smooth, i.e. that there is a neighbourhood $N$ of $y$ in $\D$ and a smooth map $\phi : N \to \C$ with $\phi(y)=0$ and $\phi(D \cap N) = \Hp \cap \phi(N)$. Both terminologies are meant to highlight when $x$ and $y$ belong to the same or different boundary components.

    The requirement that boundary components correspond to Jordan curves is to avoid punctured domains. Furthermore, note that by Koebe's theorem every finitely connected domain is conformally equivalent to a circle domain of the unit disk, i.e. a domain of the form $\D \setminus C$, $C$ a finite union of closed disjoint disks in $\D$: in particular we will often work with admissible domains having smooth boundary, because the $\sle$ measures that we construct can be transported by conformal equivalences. Importantly, these conformal equivalences can be extended to the boundary under this regularity assumption.
    \item Similarly, a domain $D'$ is called a \emph{test domain} (relatively to $(D,x,y)$ an admissible domain, crossing or non-crossing) if $D' \subset D$, $D'$ is simply-connected with $x,y \in \partial D'=:\mcalB'$ and $D'$ agrees with $D$ in a neighbourhood of $x$ and $y$. We denote by $\mcalB_l'$ (resp. $\mcalB_r'$) the part of the counterclockwise arc from $y$ to $x$ (resp. from $x$ to $y$) in $\mcalB'$ that lies in $D$ (see again Figure~\ref{fig:definitions}). We will typically consider smooth test domains, i.e. test domains $D'$ for which all boundary points are $\partial D'$-smooth.
    
    \item For $D$ a smooth simply-connected domain and distinct $x, y \in \partial D$, we denote by $g_D$ the harmonic extension of the boundary values $-\lam:= - \sqrt{\pi/8}$ (resp. $+\lam$) on $(yx)$ (resp. on $(xy)$).

    \item For $D$ a domain, we let $(\cdot,\cdot)_{\nabla(D)}$ denote the Dirichlet inner product
    \[
        (f,g)_{\nabla(D)} = \int_D \nabla f \cdot \nabla g
    \]
    for $f,g \in H^1(D)$, the Sobolev space defined as the closure of $(C^{\infty}(D), \norm{\cdot}_{\nabla(D)})$. For $\de \subset D$ a continuous curve, we denote by $(\cdot,\cdot)_{\de}$ the scalar product on $L^2(\de)$, with respect to the arclength measure on $\delta$ induced by the Euclidean metric on $\C$, that we will write $\ld{x}$. We will often abuse notation and write $(f,g)_{\delta}$ when $f$ is a distribution on $\de$ tested against a function $g$ of sufficient regularity.
\end{itemize}

\subsection{Poisson kernels and operators}\label{subsec:poissonkernel}

We review some definitions and identities between Green's function and Poisson kernels. 

\begin{itemize}
    \item The \emph{Green's function} $G_D$ in a domain $D$ at distinct $z,w \in D$ (with Dirichlet boundary conditions) is the density at $w$ of the law of a Brownian motion with speed $2$ started at $z$, killed upon exiting $D$. This is normalised so that $G_\D(0,z) = -(2\pi)^{-1}\log \abs{z}$, consistently with \cite{werner_lecture_2021, berestycki_gaussian_2024}, and so that $G_D$ is the integral kernel of $(-\Delta_D)^{-1}$ without extra multiplicative terms.
    \item Let $z \in D$ and $x$ be $\partial D$-smooth (meaning that there exists a neighbourhood $N$ of $x$ in $D$ and a smooth map $\phi : N \to \C$ with $\phi(x)=0$ and $\phi(D \cap N) = \Hp \cap \phi(N)$). We denote by $H_D(z,x)$ the \emph{Poisson kernel}, defined as the density (w.r.t. arc length) at $x$ of the measure induced by the exit point of $D$ of a Brownian motion started at $z$, i.e. for all $V \in \partial D \cap N$,
    \[
        \Pb{B_{T_D}^z \in V} = \int_V H_D(z,x') \ld{x'},
    \]
    well-defined as $x$ is $\partial D$-smooth. 
    Under this normalisation,
    \[
        H_\Hp(x+iy,0)=\inv{\pi} \frac{y}{x^2+y^2}.
    \]
    The Poisson kernel also corresponds to the normal derivative of the Green's function at the boundary, i.e.
    \[
        H_D (z,w) = \partial_{\nu_w} G_D(z,w),
    \]
    with $\nu_w$ the \textit{inward} normal vector at $w$.
    \item For distinct $x, y$ both $\partial D$-smooth, we let $H_{\partial D}(x,y)$ denote the \textit{boundary} or \textit{excursion Poisson kernel}, defined by
    \[
        H_{\partial D} (x,y) = \partial_{\nu_x} H_D(x,y) = \partial_{\nu_y} H_D(x,y),
    \]
    The boundary Poisson kernel also describes the mass of measures on Brownian excursions from the boundary of planar domains \cite[Section 3.3]{BrownianLoopSoup}: this will be important later on for probabilistic interpretations of operators defined using these kernels.
\end{itemize}

These three functions satisfy the following conformal covariance property: for $\varphi:D \to \varphi(D)$ a conformal equivalence,
\begin{align}
    G_{D}(x,y) &= G_{\varphi(D)}(\varphi(x),\varphi(y)),&&  x,y \in D \nonumber\\
    H_{D}(x,y) &= \abs{\phi'(y)} H_{\varphi(D)}(\varphi(x),\varphi(y)),&& x \in D, y \in \partial D \nonumber\\
    H_{\partial D}(x,y) &= \abs{\phi'(x)}  \abs{\phi'(y)} H_{\partial \varphi(D)}(\varphi(x),\varphi(y)),&& x,y \in \partial D \label{eq:confcovpoisson}
\end{align}
with $x$ and $\varphi(x)$ smooth on $\partial D$ and $\partial \varphi(D)$ whenever $x \in \partial D$, same for $y$, so that the quantities are well-defined.

An important family of measures on unrooted loops in domains of the complex plane are the \emph{Brownian loop measures}
\[
\mu_D^{\text{loop}}\coloneqq\int_{D} \int_{0}^{+\infty} \inv{4 \pi t^2} \mu_D^\#(z,z,t) \di t \adz,
\]
with $\mu_D^\#(z,z,t)$ corresponding to the law of a Brownian bridge from $z$ to $z$ of time duration $t$, on the event that it stays in $D$. These measures satisfy a restriction property as well as conformal invariance \cite[Proposition 6]{BrownianLoopSoup}. They are infinite, but we will make use of the fact that for $K_1,K_2$ two compacts of $D$ such that $\dist(K_1,K_2)>0$, the mass
\[
m_D(K_1,K_2)=\mu_D^{\text{loop}}(\inset{\omega \text{ loop in }D: \omega \cap K_1 \neq \emptyset, \omega \cap K_2 \neq \emptyset})
\]
of loops in $D$ that intersect both $K_1$ and $K_2$ is finite and also a conformal invariant: for $\phi:D \to \phi(D)$ a conformal equivalence, $m_D(K_1,K_2)=m_{\phi(D)}(\phi(K_1),\phi(K_2))$. 

We can also see the Poisson kernel and Green's function as integral operators. For instance, for $\mcalB_1,\mcalB_2$ disjoint smooth crosscuts in a domain $D$ with $\dist(\mcalB_1, \mcalB_2)>0$, define the bounded operator $H_{D \setminus \mcalB_1}^{\mcalB_2}:L^2(\mcalB_1) \to L^2(\mcalB_2)$ as
\[
H_{D \setminus \mcalB_1}^{\mcalB_2} f \coloneqq \int_{\mcalB_1} f(x) H_{D \setminus \mcalB_1}(\cdot,\di x),
\]
where $H_{D \setminus \mcalB_1}$ corresponds to (abusing notation) the Poisson kernel in the connected component of $D \setminus \mcalB_1$ that contains $\mcalB_2$.
In the same way, we can define integral operators $(G_D)_{\mcalB_j}^{\mcalB_k}, \inset{j,k} \subseteq \inset{1,2}$ out of the Green's function and the boundary Poisson kernel for $f \in L^2(\mcalB_j)$ and $y \in \mcalB_k$:
\begin{align*}
((G_D)_{\mcalB_j}^{\mcalB_k}f)(y) &\coloneqq \int_{\mcalB_j} G_D(y,x) f(x) \ldx,\\
((H_{\partial (D \setminus (\mcalB_1 \cup \mcalB_2))})_{\mcalB_j}^{\mcalB_k} f)(y) &\coloneqq \int_{\mcalB_1} H_{\partial (D \setminus (\mcalB_1 \cup \mcalB_2))}(y,x) f(x) \ldx \quad \text{ only if } j \neq k.
\end{align*}
Allowing for $j=k$ in the definition of the boundary Poisson kernel operator above yields an unbounded operator since $H_{\partial (D \setminus (\mcalB_1 \cup \mcalB_2))}(x,y) \asymp (x-y)^{-2}$ as $\abs{x-y} \to 0$. However, the operator
\[
(H_{\partial (D \setminus \mcalB_1)}-H_{\partial (D \setminus (\mcalB_1 \cup \mcalB_2))})_{\mcalB_1}^{\mcalB_1} f \coloneqq \int_{\mcalB_1} (H_{\partial (D \setminus \mcalB_1)}-H_{\partial (D \setminus (\mcalB_1 \cup \mcalB_2))})(\cdot,x) f(x) \ldx
\]
is bounded on $L^2(\mcalB_1)$ and well-defined with no assumption on $\mcalB_2$ other than continuity, with the difference $(H_{\partial (D \setminus \mcalB_1)}-H_{\partial (D \setminus (\mcalB_1 \cup \mcalB_2))})(x,y)$ corresponding to the mass of Brownian excursions from $x$ to $y$ in the connected component of $D \setminus \mcalB_1$ containing $\mcalB_2$ that intersect $\mcalB_2$.
Similarly, $(G_D)_{\mcalB_j}^{\mcalB_k}$ is bounded iff $j \neq k$ since $G_D(x,y) \asymp -\log \abs{x-y}$ as $\abs{x-y} \to 0$. We will often drop the superscripts when the context makes it clear.

We first gather useful identities between these operators.

\begin{lemma}[Some identities for loop measures]\label{lem:idloops}
Let $\mcalB_1, \mcalB_2$ be disjoint smooth simple curves in $D$ and distinct $x, y \in \mcalB_1$. It holds that
    \begin{equation}\label{eq:decompgreen}
    (G_D-G_{D \setminus \mcalB_2})(x,y) = \int_{\mcalB_2} H_{D \setminus \mcalB_2}(x, z) \ld{z} \int_{\mcalB_1} H_{D \setminus \mcalB_1}(z, w) \ld{w} \, G_D(w,y),
    \end{equation}
    \begin{equation}\label{eq:decomppoisson}
    \int_{\mcalB_2} H_{D \setminus \mcalB_2}(x, z) \ld{z} H_{D \setminus \mcalB_1}(z,y) = \int_{\mcalB_1} G_{D \setminus \mcalB_2}(x,w) \ld{w} \, (H_{\partial (D \setminus \mcalB_1)}-H_{\partial (D \setminus (\mcalB_1 \cup \mcalB_2))})(w, y) .
    \end{equation}

    When $\mcalB_2$ is merely assumed continuous, it holds altogether that
    \begin{equation}\label{eq:decompnoassump}
        (G_D-G_{D \setminus \mcalB_2})(x,y) = \iint_{\mcalB_1 \times \mcalB_1} G_{D \setminus \mcalB_2}(x,w) \ld{w} \, (H_{\partial (D \setminus \mcalB_1)}-H_{\partial (D \setminus (\mcalB_1 \cup \mcalB_2))})(w, z) G_{D}(z,y) \ldy
    \end{equation}
\end{lemma}
We might sometimes write more succinctly these equations as
\begin{align*}
    (G_D-G_{D \setminus \mcalB_2})_{\mcalB_1}^{\mcalB_1} &= H_{D \setminus \mcalB_2}^{\mcalB_1} H_{D \setminus \mcalB_1}^{\mcalB_2} (G_D)_{\mcalB_1}^{\mcalB_1}, \\
    H_{D \setminus \mcalB_2}^{\mcalB_1} H_{D \setminus \mcalB_1}^{\mcalB_2} &= (G_{D \setminus \mcalB_2})_{\mcalB_1}^{\mcalB_1} (H_{\partial (D \setminus \mcalB_1)}-H_{\partial (D \setminus (\mcalB_1 \cup \mcalB_2))})_{\mcalB_1}^{\mcalB_1}, \\
    (G_D-G_{D \setminus \mcalB_2})_{\mcalB_1}^{\mcalB_1} &= (G_{D \setminus \mcalB_2})_{\mcalB_1}^{\mcalB_1} (H_{\partial (D \setminus \mcalB_1)}-H_{\partial (D \setminus (\mcalB_1 \cup \mcalB_2))})_{\mcalB_1}^{\mcalB_1} (G_D)_{\mcalB_1}^{\mcalB_1}.
\end{align*}

\begin{proof}
    The left-hand side in Equation~\eqref{eq:decompgreen} corresponds to the mass of Brownian paths from $x$ to $y$ in $D$ that hit $\mcalB_2$. One can decompose such a path at its first hitting time of $\mcalB_2$, and at its following first hitting time of $\mcalB_1$: the mass of the measure disintegrated in this way is exactly the right-hand side.

    Equation~\ref{eq:decomppoisson} follows similarly by the Markov property and reversibility of the Brownian measure: the left-hand side counts Brownian loops started at $y$ that `enter' immediately $D \setminus \mcalB_1$, hit $\mcalB_2$, and come back to a point $x \in \mcalB_1$, which can be decomposed as an excursion from $\mcalB_1$ to $\mcalB_1$ in $D\setminus\mcalB_1$ hitting $\mcalB_2$ and a path to the starting point in $D \setminus \mcalB_2$.

    Finally, Equation~\ref{eq:decompnoassump} follows from both considerations, without any assumption on $\mcalB_2$ other than continuity: a Brownian path from $x$ to $y$ intersecting $\mcalB_2$ can be decomposed into a Brownian path from $x$ to some $w \in \mcalB_1$ in $D \setminus \mcalB_2$, an excursion from $w$ to some $z \in \mcalB_1$ in $D \setminus \mcalB_2$ intersecting $\mcalB_2$, and a Brownian path from $z$ to $y$ in $D$.
\end{proof}
Finally, we need a proposition from \cite{dubedat_sle_2009} that will be relevant in simplifying computations of Gaussian integrations:

\begin{proposition}[Proposition 2.2 in \cite{dubedat_sle_2009}]\label{prop:detloopsoup}
    In the setting above, the following identity holds:
    \[
        \exp(-m_D(\mcalB_1,\mcalB_2)) = \det_F(I_{\mcalB_1}-H_{D \setminus \mcalB_2}^{\mcalB_1} H_{D \setminus \mcalB_1}^{\mcalB_2}),
    \]
    where $\det_F$ stands for the Fredholm determinant. In fact, it is sufficient that $\mcalB_1$ is smooth and $\mcalB_2$ continuous, provided that we understand
    \[
        H_{D \setminus \mcalB_2}^{\mcalB_1} H_{D \setminus \mcalB_1}^{\mcalB_2} := (G_{D \setminus \mcalB_2})_{\mcalB_1}^{\mcalB_1} (H_{\partial (D \setminus \mcalB_1)}-H_{\partial (D \setminus (\mcalB_1 \cup \mcalB_2))})_{\mcalB_1}^{\mcalB_1}.
    \]
\end{proposition}

$I_{\mcalB_1}$ stands here for the identity operator on $L^2(\mcalB_1)$. We keep the notation $1_{\mcalB_1}$ for the constant function in $L^2(\mcalB_1)$ equal to one. In \cite{dubedat_sle_2009}, the result is stated for $\mcalB_2$ smooth, but we explain in Appendix \ref{app:dubgeneral} how the result follows when $\mcalB_2$ is merely continuous.


\subsection{Dirichlet energies}\label{subsec:dirichlet}

We will often consider diverging Dirichlet energies of harmonic functions that are locally constant on the boundary of a domain. Consider $D$ a smooth bounded simply connected domain again with distinct $x, y \in \partial D$ and boundary $\mcalB_l \cup \mcalB_r=(yx) \cup (xy)$ as in the notation section. Let $u_D$ denote the harmonic extension of the boundary values $-1$ on $\mcalB_l$, $+1$ on $\mcalB_r$. The Dirichlet energy $\norm{u_D}_{\nabla(D)} := \norm{\nabla u_D}_D$ diverges, but we can regularise it at the discontinuity points of $u_D$:
\[
\normreg{u_D}^2 := \lim_{\eps_1 \to 0, \eps_2 \to 0} \int_{z \in D, \abs{z-x}>\eps_1, \abs{z-y}>\eps_2} \abs{\nabla u_D}^2 \adz + \frac{4}{\pi} (\log\eps_1 + \log\eps_2).
\]
It is easy to see that the limit exists. Furthermore, it can be verified (see for instance \cite{dubedat_sle_2009}, proof of Proposition 5.2) that the regularised Dirichlet energy satisfies the following conformal covariance property, for $D$ a smooth bounded simply connected domain $\phi:D \to \phi(D)$ a bounded conformal equivalence:
\[
\exp(-\inv{2}\normreg{u_D}^2) = \abs{\phi'(x)}^{2/\pi}\abs{\phi'(y)}^{2/\pi}\exp(-\inv{2}\norm{u_{\phi(D)}}_{\nabla(\phi(D)), \reg}^2),
\]
so that up to a multiplicative constant,
\begin{equation}\label{eq:dubproof}
\exp(\normreg{u_D}^2)=H_{\partial D}(x,y)^{-4/\pi}.
\end{equation}

Nonetheless, in many relevant situations below, we will consider differences of such diverging Dirichlet energies to which it is possible to provide a meaning without renormalisation. Consider in the above set-up $D' \subset D$ another test domain such that $\dist(\mcalB_l, D \setminus D')>0$. One might formally write
\begin{align*}
\int_D \abs{\nabla u_D}^2 &= 2\int_{\mcalB_l} \partial_\nu u_D, \\
\int_{D'} \abs{\nabla u_{D'}}^2 &= 2\int_{\mcalB_l} \partial_\nu u_{D'},
\end{align*}
with $\partial_\nu$ the inwards normal derivative, and define
\[
\norm{u_{D'}}_{\nabla(D')}^2 - \norm{u_{D}}_{\nabla(D)}^2 := 2\int_{\mcalB_l} \partial_\nu (u_{D'}-u_D)
\]
where the normal derivative is well-defined and the integral is finite since $u_{D'}-u_D$ is harmonic in $D'$ with boundary values $0$ on $\mcalB_l$ and $1-u_D \geq 0$ on $\partial D' \setminus \mcalB_l$. This is of course equal to the difference of the regularised Dirichlet energies. One can also replace $u_D, u_{D'}$ by their expression involving the Poisson kernel:
\[
u_D(\cdot)=1 - 2 \int_{\mcalB_l} H_D(\cdot, y) \ldy, \quad u_{D'}(\cdot)=1 - 2 \int_{\mcalB_l} H_{D'}(\cdot, y) \ldy
\]
so that the definition above becomes
\begin{align*}
\norm{u_{D'}}_{\nabla(D')}^2 - \norm{u_{D}}_{\nabla(D)}^2 &:= 4\iint_{\mcalB_l \times \mcalB_l}(H_{\partial D} - H_{\partial D'})(z,w) \ld{z} \ld{w} \\
&= 4(1_{\mcalB_l}, (H_{\partial D} - H_{\partial D'}) 1_{\mcalB_l})_{\mcalB_l}.
\end{align*}

This also corresponds to (four times) the mass of Brownian excursions from $\mcalB_l$ that hit $D \setminus D'$. In particular, the resulting expression is conformally invariant. The difference of these Dirichlet energies can also be seen as (four times) the difference of the inverse extremal distance $\mathrm{ED}(\mcalB_l,\mcalB_r')^{-1} - \mathrm{ED}(\mcalB_l,\mcalB_r)^{-1}$, see \cite[Theorem 4-5]{ahlfors_conformal_2010}. The formula above applies also to a general admissible domain $D$ (possibly multiply connected), where $\partial D=\mcalB_- \cup \mcalB_+$ with $\mcalB_-$, $\mcalB_+$ containing resp. $(yx)$ and $(xy)$, as well as the other boundary components of $\partial D$.

\subsection{Conformal restriction {SLE}}\label{subsec:remindersle}

We start by recalling the construction of chordal $\slek$ in the upper half-plane and then present Lawler's definition of what we call conformal restriction SLE.

\subsubsection{{SLE} in simply-connected domains}

Given $\kappa>0$, let $B_{\kappa \cdot}$ be a one-dimensional Brownian motion running at speed $\kappa$, and $(g_t(z))_{t\geq 0}$ be the solution to the following chordal Loewner equation:
\[
\partial_t g_t(z) = \frac{2}{g_t(z) - B_{\kappa t}}, \quad g_0(z)=z, \quad \quad z \in \Hp.
\]
For each $z \in \Hp$, the flow above is well-defined until $T(z) = \inf \inset{t \geq 0 : g_t(z) - B_{\kappa t} = 0}$. 

The resulting hulls $K_t := \inset{z \in \Hp: T(z) \leq t}$ almost surely correspond to simple curves for $0 \leq \kappa \leq 4$, and are generated by non-crossing curves for $4 < \kappa$ (i.e. there exists $\ga$ such that $K_t$ is equal to the unbounded connected component of $\Hp \setminus \ga_t$), as proved in \cite[Theorem 5.1, 6.1]{rohde_basic_2005}. Furthermore, $\ga(t) \to \infty$ almost surely for every $\kappa>0$. The probability measure on these curves generated by the Brownian motion is called the chordal $\slek$ measure in $\Hp$ from $0$ to $\infty$ and it is scale invariant.

 $\slek$ from $x$ to $y$ in a simply-connected domain $D$ with distinct $x, y \in \partial D$ is then defined as the image of $\slek$ in $\Hp$ (from $0$ to $\infty$) by a conformal map $\varphi: \Hp \to D, \varphi(0)=x, \varphi(+\infty)=y$. We denote these probability measures $\nu_{D,\kappa}(z,w)$ and from now on understand them as measures on curves up to (positive) time reparametrisation (since the conformal map $\varphi$ above is not unique, but there remains freedom in precomposing by a dilation).

The one-parameter family of measures $\inset{\nu_{D, \kappa}(z,w)}_{(D,z,w)}$ in fact describes the only probability measures on curves from $z$ to $w$ in $D$ satisfying conformal invariance (clear from the construction) and a conformal Markov property, that states that the conditional law of $\ga([t,+\infty))$ under $\nu_{D,\kappa}(z,w)$ given $\ga_t$ is equal to $\nu_{D \setminus \ga_t,\kappa}(\ga(t),w)$ (see \cite{schramm_scaling_2000}). For the rest, for readability let us leave out the dependence on $\kappa$ in the notation, and restrict ourselves to the range $\kappa \leq 4$ where the curves are simple.

\subsubsection{Conformal restriction {SLE}}

In multiply connected domains, the two properties above are not sufficient to characterise $\sle$\footnote{For instance, one could define a different SLE type of curve in a conformal annulus, that would behave like $\sle_2$ when the conformal modulus is greater than one and like $\sle_4$ otherwise. The conformal invariance requirement is too weak capture this change of behaviour of the curve because it cannot compare the law of the curve for different conformal classes of domains.}. However, in the simply connected case, these curves satisfy an extra property, called the restriction property or boundary invariance rule \cite{lawler_conformal_2003} (see also \cite[Proposition 2.1]{lawler_partition_2009}): for $K \subset \Hp$ a compact hull such that $\dist(K,0)>0$ and $\Hp \setminus K$ is (connected and) simply connected, it holds that $\nu_{\Hp\setminus K}(0,+\infty) \ll \nu_\Hp(0,+\infty)$ and that
\begin{equation}\label{eq:restrictionslehp}
\frac{\di \nu_{\Hp\setminus K}(0,+\infty)}{\di \nu_\Hp(0,+\infty)}(\eta) = \ind{\eta \subset \Hp \setminus K} \abs{\varphi'(0)}^{-h} \exp(\frac{\cc}{2} m_{\Hp}(\eta, K)),
\end{equation}
where $\varphi:\Hp\setminus K \to \Hp$ is the conformal equivalence such that $\varphi(0)=0$ and $\varphi(z) = z + o(1)$ as $z \to \infty$. We note that the term $\abs{\varphi'(0)}$ corresponds to the probability that a Brownian excursion from $0$ to $\infty$ in $\Hp$ does not hit $K$ \cite[Proposition 4.1]{lawler_conformal_2003}: similarly, for $D$ a simply connected admissible domain and $D'$ a test domain,
\begin{equation}\label{eq:restrictionsled}
\frac{\di \nu_{D'}(x,y)}{\di \nu_D(x,y)}(\eta) = \ind{\eta \subset D'} \paren{\frac{H_{\partial D'}(x,y)}{H_{\partial D}(x,y)}}^{-h} \exp(\frac{\cc}{2} m_{D}(\eta, D \setminus D')),
\end{equation}
where $H_{\partial D'}(x,y)/H_{\partial D}(x,y)$ is again the probability that an excursion from $x$ to $y$ in $D$ remains in $D'$.
It was suggested in \cite{lawler_partition_2009} to use this property to characterise $\sle$ in general domains, provided that we consider $\sle$ as measures with mass $Z_{D}(x,y)=Z_{D,\kappa}(x,y)$. Note that some regularity of $\partial D$ at $x$ and $y$ is required to define this mass or partition function: but this is not restrictive at the level of the corresponding probability measures, as any (finitely connected) domain is conformally equivalent to one with smooth boundary around the images of $x$ and $y$, so that the $\sle$ probability measure can be defined there and brought back to the initial domain.

\begin{definition}[Conformal restriction {SLE}, \cite{lawler_defining_2011}]\label{def:sle}
    $\sle$ is the unique family of measures $\inset{\nu_D(x,y)}_{(D,x,y)}$ on curves (modulo time-reparametrisation) in $D$ a general admissible domain from $x$ to $y$ of total mass $Z_D(x,y)=\norm{\nu_D(x,y)}$ normalised so that $Z_\Hp(0,+\infty)=1$, satisfying the following:
    \begin{itemize}
        \item \textbf{Conformal covariance} If $\varphi:D \to D'$ is a conformal mapping, then the pushforward measure $\varphi_\ast \nu_D(x,y)$ satisfies
        \begin{equation}\label{eq:confcov}
            \varphi_\ast \nu_D(x,y) = \abs{\varphi'(x)}^h \abs{\varphi'(y)}^h \nu_{D'}(f(x),f(y)).
        \end{equation}
        In particular, at the level of probability measures, denoting $\nu_D^\#(x,y) := \nu_D(x,y)/Z_D(x,y)$ if $Z_D(x,y)=\norm{\nu_D(x,y)}<+\infty$, conformal invariance holds:
        \[
            \varphi_\ast \nu_D^\#(x,y) = \nu_{D'}^\#(f(x),f(y)).
        \]
        \item \textbf{Domain Markov property} If $Z_D(x,y)<+\infty$, then the conditional probability measure of the remainder $\eta([t,+\infty))$ of a curve $\eta$ with respect to $\nu_D^\#(x,y)$ conditional on $\eta_t$ is given by $\nu_{D \setminus \eta_t}^\#(\eta(t),y)$.
        \item \textbf{Boundary perturbation/restriction property} Let $D'\subset D$ be admissible domains agreeing in a neighbourhood of $x$ and $y$. Then $\nu_{D'}(x,y) \ll \nu_D(x,y)$, with Radon-Nikodym derivative given by
        \begin{equation}\label{eq:restriction}
            \frac{\di \nu_{D'}(x,y)}{\di \nu_D(x,y)}(\eta)=\ind{\eta \subset D'} \exp(\frac{\cc}{2}m_D(\eta,D\setminus D')).
        \end{equation}
    \end{itemize}
\end{definition}

Note that in the case of simply connected domains, it is consistent with Equations~\eqref{eq:restrictionslehp} and \eqref{eq:restrictionsled} to define $Z_D(x,y)=H_D(x,y)^h$ when $x,y \neq \infty$, and $Z_D(x,y)=\abs{\phi'(x)}^h$ when $y=\infty$, where $\varphi:D\to \Hp$ is the conformal equivalence with $\varphi(x)=0$ and $\varphi(z)=z+o(1)$ as $z \to \infty$.

In \cite[Section 4]{lawler_defining_2011}, these measures are constructed in any general admissible domain, from a Carathéodory argument: the restriction of $\nu_D(x,y)$ to curves contained in a general test domain $D'$ is prescribed by the boundary perturbation rule, and the restriction to different test domains are compatible thanks to the additive structure of the loop measure. 
This justifies the denomination of `test domains': to identify a measure on curves in a domain $D$ as the $\sle$ measure it is sufficient to verify Equation~\eqref{eq:restriction} for $D' \subset D$ simply connected. As a matter of fact, it is enough to take $D'$ with smooth boundary, as the identity then holds for all subdomains by approximation (see \cite{lawler_conformal_2003} for more details about approximations of hulls by smooth hulls). Note that uniqueness also follows from the construction.

We further stress that some conditions in the definition above are redundant: the domain Markov property and conformal covariance for $\sle$ in general domains is a consequence of the fact that these properties hold on simply connected domains, and are inherited through the boundary perturbation rule (provided that the partition function is finite: see again \cite[Section 4]{lawler_defining_2011} for all the details). Reversibility of the $\sle$ is also inherited from the simply connected case, which was proved in \cite{zhan_reversibility_2008}.

On the other hand, finiteness of the total mass i.e. the partition function is less clear. In \cite{lawler_defining_2011}, it is shown to be finite when $\kappa\leq 8/3$, or when $D$ is a conformal annulus. In this paper, we construct explicitly the $\sle_4$ measure in all topological cases and the $\slek$ measure for $8/3 < \ka \leq 4$ in the non-crossing case, and provide an expression for their partition function that proves finiteness.


\section{\texorpdfstring{$\sle_4$}{SLE4} as a mixture of level-lines of Gaussian free fields}\label{sec:sle4}

In this section we provide an explicit construction of the $\sle$ measure in the specific case of $\kappa=4$, exploiting its relation with the Gaussian free field: this is possible both in the non-crossing and crossing case. First, we collect some classical facts about Gaussian free fields and $\sle_4$.

\subsection{Preliminaries}
\subsubsection{Gaussian free field}

The (Dirichlet, with zero boundary conditions) Gaussian Free Field (GFF) in a domain $D$ is a Gaussian process $\Phi$ on $H_0^1(D)$, with mean zero and covariance function
\[
\E{\Phi_{f_1} \Phi_{f_2}} = \iint_{D \times D} f_1(x) G_D(x,y)f_2(y) \ad{x} \ad{y} < + \infty,
\]
where $G_D$ is the Green function previously introduced.
It is possible to define the GFF as a random element of $H^{-\eps}(D)$ for any $\eps > 0$ \cite{werner_lecture_2021, dubedat_sle_2009}); in particular, this justifies the writing of $\Phi_f$ as $(\Phi,f)_D$. For our purposes it will suffice to see the GFF as an element of $H^{-1}(D)$. For $g$ a harmonic function in $D$, the GFF with boundary condition $g$ is understood as $\Phi^g := \Phi^0 + g$, where $\Phi^0$ is a zero-boundary GFF in $D$. From now on we let $\Phi_D^g$ be our notation for a GFF in $D$ with mean $g$.

For $f \in H_0^1(D)$, $\Phi_D^0+f$ is also a zero-boundary GFF under the law that is also absolutely continuous w.r.t. that of $\Phi_D^0$, with Radon-Nikodym derivative given by
\[
\exp((\Phi_D^0, f)_{\nabla(D)} - (f,f)_{\nabla(D)}),
\]
where $( \cdot, \cdot)_{\nabla(D)}$ is the Dirichlet inner product and $(\Phi_D^0, f)_{\nabla(D)} := (\Phi_D^0, -\Delta f)$. \footnote{Note that it is possible to naturally make sense of the GFF tested against functions in $H^{-1}(D)$, such as $-\Delta f \in H^{-1}(D)$: see \cite[Section 1.8]{berestycki_gaussian_2024}.} This in fact only holds when $f \in H_0^1(D)$, identifying the latter as the Cameron-Martin space of the GFF.

The GFF also enjoys a natural domain Markov property \cite[Chapter 4]{werner_lecture_2021}. We will use it here in the following set-up. For $D$ a domain and $\delta$ a smooth crosscut splitting $D$ in $D_1$ and $D_2$, we can write
\[
\Phi_D^0 = \Phi_{D_1}^0 + v_{\Phi_D^0\restr{\delta}} + \Phi_{D_2}^0
\]
where $v_{\Phi_D^0\restr{\delta}}$ is a random harmonic function in $D \setminus \delta$, and all summands on the right-hand side are independent. In fact, one can define (see \cite[Section 4.3.2]{dubedat_sle_2009}) a pathwise trace of $\Phi_D^0$ on $\delta$ as $\Phi_D^0\restr{\delta} \in H^{-1}(\de)$, whose law is a centred Gaussian process of covariance $G_D \restr{\de \times \de}$, so that $v_{\Phi_D^0\restr{\delta}}$ defines the harmonic extension of $\Phi_D^0\restr{\delta}$ to $D\setminus \de$ (which is indeed harmonic as the Poisson kernel is smoothing).\\


An important class of random subsets of $D$, commonly called local sets, are those with which the GFF satisfies a strong Markov property.

\begin{definition}[Local sets]
    Given a triple $(\Phi=\Phi_D, A, \Phi_A)$ with $\Phi_D$ a GFF in $D$, $A \subset \overline{D}$ a random closed set and $\Phi_A$ a random distribution that can be seen as a harmonic function $h_A$ when restricted to $D \setminus A$, we say that $A$ is a local set for $\Phi$, or that $(\Phi, A, \Phi_A)$ is a local set coupling, if conditionally on $(A, \Phi_A)$, $\Phi - \Phi_A$ is a zero-boundary GFF in $D \setminus A$.
\end{definition}

Changing the measure with a shift of the GFF given by some $f \in H_0^1(D)$ as in the Cameron-Martin theorem is also compatible with local sets, as shows the following lemma from \cite{aru_extremal_2022}:

\begin{lemma}[Lemma 3.1 in \cite{aru_extremal_2022}]\label{lem:rnderivativecle4}
    Consider $(\Phi:=\Phi_D^0,A,\Phi_A)$ a local set coupling, $\tilde \Phi = \Phi - f$, and $f_A$ the orthogonal projection of $f$ in $H_0^1(D)$ on the subspace of functions that are harmonic in $D \setminus A$. Set $\tilde \Phi_A := \Phi_A - f_A$. Then $(\tilde \Phi, A, \tilde \Phi_A)$ is a local set coupling and the law of $(A, \Phi_A)$ under $\Phi$ is absolutely continuous with respect to the law of $(A, \tilde \Phi_A + f_A) = (A, \Phi_A)$ under $\tilde \Phi$, with Radon-Nikodym derivative given by
    \[
    \exp((\Phi_A,f)_\nabla - \inv{2}(f_A,f_A)_\nabla). 
    \]
\end{lemma}

Lastly, we will later need to compute the expectation of expressions similar to that appearing in the previous lemma, for $f$ a random function, normally distributed. The following lemma provides a formula that will be relevant in this computation.

\begin{lemma}[Proposition 1.2.8 in \cite{da_prato_second_2002}]\label{lem:dub}
    Let $H$ be a Hilbert space with $Q$ a positive, trace class operator, $M$ symmetric such that $Q^{1/2}MQ^{1/2} < I$, $m \in H$. Denote by $N_Q$ the centred Gaussian measure on $H$ with covariance $Q$. Then
    \begin{align*}
    \Ef{N_Q}{\exp(\inv{2} \scal{Mh,h}_H + \scal{m,h}_H)} \quad & \\
    = \paren{\det_F(1-Q^{1/2}MQ^{1/2})}^{-1/2}& \exp(\inv{2} \norm{(1-Q^{1/2}MQ^{1/2})^{-1/2} Q^{1/2} m}_H^2).
    \end{align*}
\end{lemma}

\subsubsection{Coupling between $\sle_4$ and the GFF}\label{subsec:coupling}

In simply connected domains, the coupling between GFF and $\sle_4$ was discovered in \cite{schramm_contour_2013, dubedat_sle_2009}, see also \cite[Chapter 5]{werner_lecture_2021}. To state this result, recall that we denote by $g_D$ the harmonic function with boundary values $-\lam$ (resp. $+\lam$) on $(yx)$ (resp. on $(xy)$), and by $\Phi_D^{g_D}=\Phi_D^0 + g_D$ a GFF in $D$ with mean given by $g_D$.

\begin{theorem}[$\sle_4$ is a local set for the GFF]\label{thm:sle4gffsimplyconnected}
    Let $D$ be a simply connected domain with distinct $x, y \in \partial D$. There exists a coupling $(\Phi_D^{g_D}, \eta, \Phi_\eta)$ with $\eta \sim \nu_{D,4}^\#(x,y)$ an $\sle_4$ curve from $x$ to $y$ in $D$ that is a local set coupling, with $\Phi_\eta$ given by $h_\eta$ on $D \setminus \eta$, $h_\eta(z):= - \lam + 2 \lam \mathbf{1}(z \in (D\setminus\eta)_{\text{right}})$. Furthermore, in this coupling, $\eta$ is a deterministic function of $\Phi_D^{g_D}$.
\end{theorem}

In other words, there is a local set coupling $(\Phi_D^{g_D}, \eta, h_\eta)$ in which the marginal law of $\eta$ is that of an $\sle_4$ from $x$ to $y$ and such that conditionally on $\eta$, $\Phi_D^{g_D} = \Phi_{D \setminus \eta, \text{left}}^{-\lam} + \Phi_{D \setminus \eta, \text{right}}^{+\lam}$, a sum of two independent GFFs on each side of $D \setminus \eta$. When the domain is no longer simply connected, or the harmonic function $g_D$ is replaced with a bounded harmonic function with different boundary conditions away from the target points $x$ and $y$, it is still possible to make sense of a level-line of the GFF.

\begin{proposition}[General level-lines, {{\cite[Lemma 15]{aru_bounded-type_2019}}}]
    Let $(D,x,y)$ be a non-crossing admissible domain and $g$ a bounded harmonic function in $D$ with Dirichlet boundary conditions $-\lam$ (resp. $+\lam$) on $(yx)$ (resp. $(xy)$), arbitrary on other parts of $\partial D$. Then there exists a unique law on random simple curves $\eta$ in $D$ (modulo time-reparametrisation, up to its exit time $t_\eta$ of $D$) with $\eta(0)=x$ that can be coupled with a GFF $\Phi_D^g$ such that $(\Phi_D^g,\eta,\Phi_\eta)$ is a local set coupling, and $\Phi_\eta$ corresponds to the unique harmonic function in $D \setminus \eta$ with boundary values given by $g$ on $\partial D$ and $-\lam$ (resp. $+\lam$) on the left side of $\eta$ (resp. on the right side of $\eta$). In this coupling $\eta$ is a deterministic function of $\Phi_D^{g}$ and we call it the generalised level-line of $\Phi_D^g$. 
\end{proposition}

Furthermore, the probability that $\eta(t)$ converges to $y$ as $t \to t_\eta$ is equal to $1$ if $g$ is piecewise constant and $\abs{g} \geq \lam$ on each connected component of $\partial D \setminus \inset{x,y}$ \cite[Lemma 3.1]{powell_level_2017} (see also \cite[Lemma 16]{aru_bounded-type_2019}, or Appendix~\ref{app:existencelevelline} that retraces this argument in the case of crossing domains). This will always be the case for us, so that the level-lines correspond to simple curves from $x$ to $y$ in $D$ (up to reparametrisation). 

In the case of $(D,x,y)$ a crossing domain, generalised level-lines can also be defined for GFFs $\Phi_D^g$, where this time $g$ is harmonic multivalued with boundary values locally constant with a change of $\pm 2 \lam$ around $x$ and $y$ on $\partial D$ (see Section~\ref{subsec:crossingsle4}). The proof is very similar to \cite[Lemma 16]{aru_bounded-type_2019}, but for the sake of completeness we include a more thorough explanation in Appendix~\ref{app:existencelevelline}.

\subsection{Constructing the non-crossing \texorpdfstring{$\sle_4$}{SLE4} measure}\label{subsec:noncrossingsle4}

In this section, we construct a measure on curves joining two points in the same boundary component of a non-crossing admissible domain $(D, x, y)$ as defined in Section~\ref{subsec:notations}. We start by explaining the set-up, then state and prove Theorem~\ref{thm:noncrossing} and finish the section with some additional remarks.

\subsubsection{Set-up and statement}

For $\bfb = (b_1, \dots, b_n) \in \inset{\pm\lam}^n$, we consider the set of curves from $x$ to $y$ in $D$ leaving the components of $\mcalB$ corresponding to the negative components of $\bfb$ to the left, the others to the right. Formally, we define the events
    \begin{align*}
        \mcalA_{\mathbf{b}} := \{\eta \text{ curve } x \to y \text{ in } D \, | \, \forall j & \text{ s.t. } b_j=-\lam, \mcalB_j\subset(D \setminus \eta)_{\text{left}} \text{ and } \forall j \text{ s.t. } b_j=+\lam, \mcalB_j\subset(D \setminus \eta)_{\text{right}}\},
    \end{align*}
with curves here understood modulo time-increasing reparametrisation. These events partition $\inset{\eta \subset D}$ since $\mcalA_\bfb \cap \mcalA_{\bfb'}=\emptyset$ for $\bfb' \neq \bfb$, and any simple curve from $x$ to $y$ in $D$ must belong to $\mcalA_{\bfb'}$ for some $\bfb' \in \inset{\pm \lam}^n$.

Let also $\mcalB_\bfb^-$ (resp. $\mcalB_\bfb^+$) correspond to the union of $\mcalB_l$ (resp. $\mcalB_r$) and of the boundary components $\mcalB_j$ such that $b_j = - \lam$ (resp. $b_j = + \lam$).

\begin{figure}[!ht]
	   \centering
	   \includegraphics[scale=0.8]{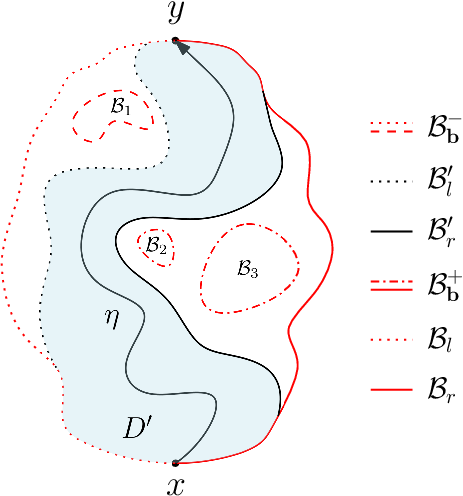}
        \caption{An example of a domain with the definitions above, together with a test domain $D'$ shaded in light blue. Here the curve $\eta$ depicted in the middle is in $\mcalA_\bfb$ for $\bfb = \inset{-\lam, + \lam, + \lam}$. In particular, the boundary values of $g_\bfb$ are $-\lam$ and $+\lam$ exactly on the parts of $\partial D$ that lie on the left and on the right of $\eta$, respectively.}
        \label{fig:definitions}
\end{figure}

\begin{figure}[!ht]
	   \centering
	   \includegraphics[scale=0.5]{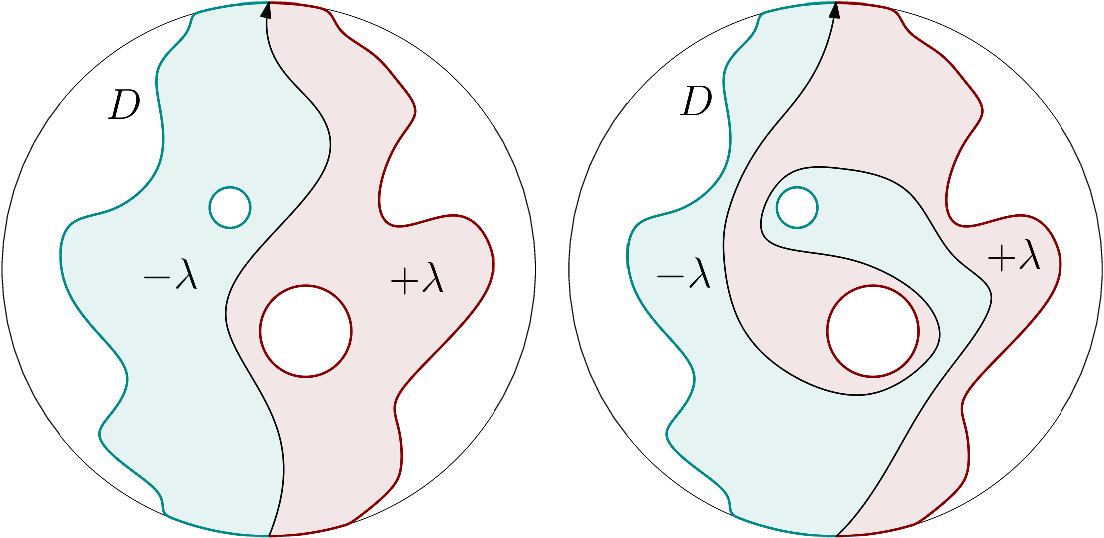}
        \caption{An example of two non path-homotopic curves in the event $\mcalA_{\inset{-\lam,+\lam}}$. In general, there are infinitely (countably) many homotopy classes of distinct curves corresponding to every event $\mcalA_\bfb$.}
        \label{fig:inequivalent_configs}
\end{figure}

Note that the events $\mcalA_\bfb$ do not quite fix the homotopy class of the curve, see Figure~\ref{fig:inequivalent_configs}. However, for $D'$ a test domain relative to $D$, there is a unique $\bfb \in \inset{\pm \lam}^n$ such that $\inset{\eta \in D'} \subset \inset{\eta \in \mcalA_\bfb}$.

We define $\mbbP_D^{g_\bfb}$  as the law of the generalised level line of a GFF $\Phi_D^{g_\bfb}$ in $D$ with boundary conditions given by $g_\bfb$, the harmonic function with boundary values $-\lam$ on $\mcalB_\bfb^-$, $+\lam$ on $\mcalB_\bfb^+$ (the dependency in $x,y$ is implicit). The important picture to keep in mind is that on the event $\mcalA_\bfb$ and conditionally on $\eta$ the level-line of the GFF, the latter is equal in law to a GFF in $D \setminus \eta$ with piecewise constant mean: $-\lam$ on the left of $\eta$, $+\lam$ on its right. That would not be true if $\eta \notin \mcalA_\bfb$.

\begin{theorem}\label{thm:noncrossing}
    Let $(D,x,y)$ be a non-crossing admissible domain. The measure $\nu_D$ on simple curves from $x$ to $y$ in $D$ defined as
    \begin{equation}\label{eq:defnoncrossingsle4measure}
        \nu_D := \sum_{\bfb \in \inset{\pm \lam}^n} \ind{\eta \in \mcalA_\bfb}\mbbP_D^{g_\bfb} \cdot Z_{D,\bfb},
    \end{equation}
    where
    \begin{equation}\label{eq:defz4}
        Z_{D,\bfb} = \exp(- \inv{2} \normreg{g_\bfb}^2)
    \end{equation}
    is the conformal restriction measure of $\sle_4$ in $(D,x,y)$.
\end{theorem} 

An immediate corollary is the finiteness of the partition function, with an explicit expression.

\begin{corollary}\label{cor:partitionfunctionnoncrossing4}
    The partition function of $\sle_4$ in a non-crossing admissible domain $(D,x,y)$ can be expressed as
    \begin{align*}
        Z_D(x,y):= \sum_{\bfb \in \inset{\pm \lam}^n}\exp(-\inv{2} \normreg{g_\bfb}^2) \cdot \mbbP_D^{g_\bfb}(\mcalA_\bfb)
    \end{align*}
    and it is finite.
\end{corollary}
Alternatively, Equation~\eqref{eq:dubproof} can be used to write
\[
Z_D(x,y)=\sum_{\bfb \in \inset{\pm \lam}^n}\exp(\inv{2} (\norm{g_{D''}}_{\nabla(D''), \reg}^2-\norm{g_\bfb}_{\nabla(D), \reg}^2)) H_{D''}(x,y)^{\sinv{4}} \mbbP_D^{g_\bfb}(\mcalA_\bfb)
\]
from which conformal covariance is perhaps best read, given that the the probability term and the difference of Dirichlet energies is conformally invariant (see Section~\ref{subsec:dirichlet}) and that the boundary Poisson kernel is conformally covariant as in Equation~\eqref{eq:confcovpoisson}.

\subsubsection{Proof of the theorem}
Before going to the proof of Theorem~\ref{thm:noncrossing}, we make some simple observations about a GFF in $D$ and its coupling with a level-line. We fix $\bfb \in \inset{\pm \lam}^n$. Let $(\Phi_D, \eta)$ be a local set coupling with $\Phi_D:= \Phi_D^{g_\bfb}$ a GFF in $D$ with boundary values given by $g_\bfb$. On the event $\mcalA_\bfb$, the harmonic part $h_{\eta}$ is two-valued and reads
\begin{equation}\label{eq:hnt}
h_{\eta}(z) = (- \lam + 2 \lam \ind{(D\setminus\eta)_{\text{right}}}(z)).
\end{equation}
In other words, conditionally on $\eta$, $\Phi_D$ is equal in law to the sum of two independent GFFs $\Phi_{(D\setminus\eta)_{\text{left}}}^0 - \lam, \Phi_{(D\setminus\eta)_{\text{right}}}^0 + \lam$ on the left (resp. right) side of $\eta$ in $D$.
We write, conditionally on $\eta$ and on the event $\mcalA_\bfb$,
\[
\Phi_D = \Phi_{D\setminus\eta}^{0} + h_\eta,
\]
where $\Phi_{D\setminus\eta}^{0}$ should be understood as $\Phi_{(D\setminus\eta)_{\text{left}}}^0 + \Phi_{(D\setminus\eta)_{\text{right}}}^0$.
Let us now consider $D' \subset D$ a test domain with $\inset{\eta \subset D'} \subset \mcalA_\bfb$ (see again Fig~\ref{fig:definitions}). Using the standard Markov property on $\Phi_{D\setminus\eta}^{0}$ on the event that $\eta \subset D'$, we can further expand:
\[
\Phi_D\restr{D'} = \Phi_{D'\setminus\eta}^{0} + u_{\mcalB_l'\cup \mcalB_r'}^{D'\setminus \eta} + h_\eta\restr{D'},
\]
where $u_{\mcalB'}^{D'\setminus \eta}$ is a harmonic function in $D'\setminus\eta$ corresponding to the harmonic extension of the restriction $\xi_l$ (resp. $\xi_r$) of $\Phi_{D \setminus \eta}^0$ on $\mcalB_l'$ (resp. $\mcalB_r'$), and of zero on the other parts of the boundary (including $\eta$).
In other words, since $\xi=(\xi_l,\xi_r)$ is distributed like a centred Gaussian process on $\mcalB_l' \cup \mcalB_r'$ whose covariance matrix is diagonal with elements $(G_{D \setminus \eta})\restr{\mcalB_l' \times \mcalB_l'}$ and $(G_{D \setminus \eta})\restr{\mcalB_r' \times \mcalB_r'}$, conditionally on $\eta$ still and on the event $\inset{\eta \subset D'} \subset \mcalA_\bfb$, we can alternatively write $\Phi_D \restr{D'}$ as $\Phi_{D'\setminus \eta}^0 + \bar h_\eta$, with $\bar h_\eta$ a random harmonic function in $D'\setminus \eta$ with boundary values
\begin{equation*}
\bar h_\eta = \begin{dcases}
    - \lam                                                           & : \text{ on the left-side of } \eta \\
    + \lam                                                          & : \text{ on the right-side of } \eta \\
    \xi_l \sim \mcalN(- \lam, G_{D \setminus \eta}\restr{\mcalB_l'}) & : \text{ on } \mcalB_l' \\
    \xi_r \sim \mcalN(+ \lam, G_{D \setminus \eta}\restr{\mcalB_r'}) & : \text{ on } \mcalB_r'
\end{dcases}.
\end{equation*}

We consider, similarly, another local set coupling $(\tilde \Phi_{D'}, \tilde \eta)$ in $D'$ with $\tilde \Phi_{D'} = \tilde \Phi_{D'}^{g_{D'}}$, so that conditionally on $\tilde \eta$ we can write $\tilde \Phi_{D'} = \tilde \Phi_{D' \setminus \tilde \eta}^{0} + h_{\tilde \eta}$, with
\[
h_{\tilde \eta}(z) = (- \lam + 2 \lam \ind{(D'\setminus \tilde \eta)_{\text{right}}}(z)).
\]
For computations, it will sometimes be easier to work with a GFF with zero-boundary conditions, coupled with a level-line such that the resulting harmonic part $h_{\tilde \eta}$ is no longer only a two-valued function, but includes a subtraction by $g_{D'}$. The marginal law of $\tilde \eta$ in this coupling will be denoted $\mbbP_{D'}^{g_{D'}}$, and in this case is equal to the $\sle_4$ measure $\nu_D^\#(x,y)$ by Theorem~\ref{thm:sle4gffsimplyconnected}.

On the other hand, the marginal law of $\eta$ in the coupling $(\eta, \Phi_D)$ as presented above on the event $\inset{\eta \subset D'}$ and \textit{given the values of} $\xi$ is denoted $\mbbP_{D'}^{g_{\bfb}, \xi}$. It therefore corresponds to the law of the level-line of a GFF in $D'$ whose mean is the harmonic extension of $\xi$ on $\mcalB_l' \cup \mcalB_r'$ and $-\lam$ (resp. $+\lam$) on $\mcalB_l\setminus \mcalB_l'$ (resp. $\mcalB_r\setminus \mcalB_r'$).\\

We now turn to the proof of Theorem~\ref{thm:noncrossing}, which essentially consists of leveraging the properties of the GFF to compute the Radon-Nikodym derivatives for the laws of its level-lines.
The key computation is contained in Proposition~\ref{prop:computderiv}. We start with a simple measure-theoretic lemma.

\begin{lemma}\label{lem:rnderivativedisintegrationlemma}
    Let $X, Y_1,Y_2$ be random variables on the same probability space, taking values in a Polish space $E$, with associated distributions $\mbbP_X, \mbbP_{Y_1}, \mbbP_{Y_2}$ respectively. Denote also by $\mbbP_{Y_1 | Y_2 = y_2}$ a version of the conditional law of $Y_1$ on ${Y_2 = y_2}$ (similarly for $\mbbP_{Y_2 | Y_1 = x}$) and assume that $\mbbP_X \ll \mbbP_{Y_1 | Y_2 = y_2}$, for $\mbbP_{Y_2}$-almost every $y_2 \in E$.
    Then the following equality between Radon-Nikodym derivatives holds $\mbbP_{X}$-a.s.
    \[
        \frac{\di \mbbP_X}{\di \mbbP_{Y_1}} (x) = \Eb{Y_2 | Y_1 = x}{\frac{\di \mbbP_X}{\di \mbbP_{Y_1 | Y_2 = y}}(x) \bigg|_{y=Y_2}}.
    \]
\end{lemma}

\begin{proof}
    See Appendix~\ref{app:prooflemma}.
\end{proof}

\begin{proof}[Proof of Theorem~\ref{thm:noncrossing}]
    We proceed by induction on $n$, the number of connected components of $\mcalB$ that do not contain $x$ and $y$ \footnote{One might wonder why we use induction and do not do the computations directly. The reason is to avoid dealing with two boundaries at once as it makes some of the computations presented below more lengthy and involved.}. For clarity we temporarily denote the measures in the statement of the theorem as $\tilde \nu_D$, as we have not yet proved that they correspond to the $\sle_4$ measures.
    
    When $n=0$, Equation~\eqref{eq:defnoncrossingsle4measure} defines $\tilde \nu_D$ as the law of the level line of a GFF in $D$ with boundary conditions $-\lam$ on the counterclockwise arc $y \to x$, $+\lam$ on the other, which corresponds to $\sle_4$. Furthermore, the expression of $Z_D$ in Equation~\eqref{eq:defz4} agrees with that of the partition function of $\sle_4$ in simply connected domains since from Equation~\eqref{eq:dubproof} it follows that
    \begin{equation}\label{eq:idZ4dub}
        \exp(-\inv{2} \normreg{g_{D}}^2)=H_D(x,y)^{\sfrac{1}{4}}.
    \end{equation}
    
    Let now $n\geq 1$ and $D'' \subset D$ be a test domain: because $D''$ is simply connected by assumption, there exists $\bfb'=(b_1',\dots,b_n') \in \inset{\pm \lam}^n$ such that $\inset{\eta \subset D''} \subset \mcalA_{\bfb'}$. Up to relabeling the components of $\mcalB$ we may assume that $\bfb' =\inset{-\lam, \dots, -\lam,+\lam,\dots,+\lam}$ with $1\leq k \leq n$ such that $b_k'=-\lam, b_{k+1}'=+\lam$ (assuming w.l.o.g. that $k \geq 1$, otherwise reasoning on the right of $D''$ instead), so that if $D'$ is the union of $D''$ and the component of $D \setminus D''$ on the left of $\eta$, we have by induction that
    \[
        \frac{\di \nu_{D''}}{\di \tilde \nu_{D'}}(\eta) = \ind{\eta \subset D''} \exp(\inv{2}m_{D'}(\eta, D' \setminus D'')).
    \]

    \begin{figure}[!ht]
	   \centering
	   \includegraphics[scale=0.5]{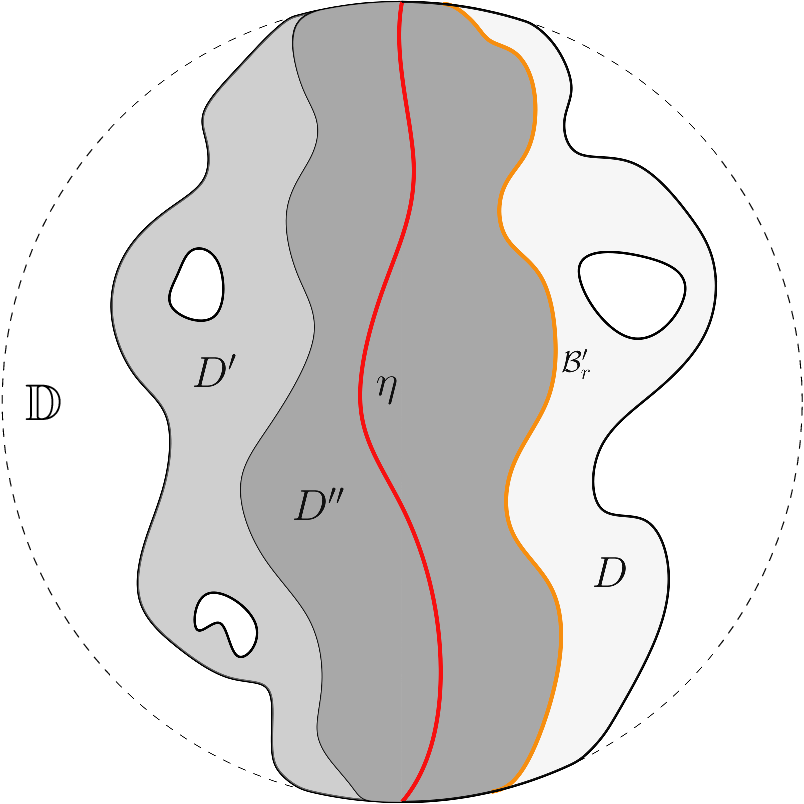}
        \caption{$D''$ corresponds to the darkest shade, $D'$ to the two darker shades (left of orange curve $\mcalB_r'$) and $D$ to the three shades.}
        \label{fig:three_domains}
    \end{figure}

    We have to show that $\tilde \nu_{D'}$ is absolutely continuous w.r.t. $\tilde \nu_D$, with Radon-Nikodym derivative given by
    \begin{equation}\label{eq:rnderiv4tocompute}
        \frac{\di \tilde \nu_{D'}}{\di \tilde\nu_{D}}(\eta) = \sum_{\substack{\bfb = (\bfb_k, +\lam, \dots, +\lam) \\ \bfb_k \in \inset{\pm \lam}^{k}}} \ind{\eta \in \mcalA_\bfb} \frac{\di \mbbP_{D'}^{g_{\bfb_k}}}{\di \mbbP_{D}^{g_{\bfb}}}(\eta) \frac{Z_{D',\bfb_k}}{Z_{D,\bfb}}
        = \sum_{\substack{\bfb = (\bfb_k, +\lam, \dots, +\lam) \\ \bfb_k \in \inset{\pm \lam}^{k}}} \ind{\eta \in \mcalA_{\bfb}, \eta \subset D'} \exp(\inv{2}m_{D}(\eta, D \setminus D')).
    \end{equation}
    Indeed, by the additive property of the loop measure we would then obtain
    \begin{align*}
        \frac{\di \nu_{D''}}{\di \tilde\nu_{D}}(\eta) = \frac{\di \nu_{D''}}{\di \tilde\nu_{D'}}(\eta)\frac{\di \tilde \nu_{D'}}{\di \tilde\nu_{D}}(\eta) &= \ind{\eta \subset D'} \exp(\inv{2}m_{D'}(\eta, D' \setminus D'') + \inv{2} m_{D}(\eta, D \setminus D')) \\
        &= \ind{\eta \subset D'} \exp(\inv{2}m_{D}(\eta, D \setminus D'')).
    \end{align*}
    To carry out this computation, we use Lemma~\ref{lem:rnderivativedisintegrationlemma} with $\mbbP_X=\mbbP_{D'}^{g_{\bfb}}$ (now abusing notation to drop the superscript $\bfb_k$), $\mbbP_{Y_1}=\mbbP_{D}^{g_{\bfb}}$, and $\mbbP_{Y_2}$ corresponding to the law of $\xi := T_{\mcalB_r'}\Phi_D^0$, the restriction of the GFF to $\mcalB_r'$. As noted earlier, conditionally on $\eta\sim \mbbP_D^{g_{\bfb}}$ and $\inset{\eta \subset D'}$, we have that $\xi \sim \mcalN(0, G_{D \setminus \eta}\restr{\mcalB_r'})$.
    We can then write
    \[
        \frac{\di \mbbP_{D'}^{g_{\bfb}}}{\di \mbbP_{D}^{g_{\bfb}}} = \Ef{\mcalN(0, G_{D \setminus \eta}\restr{\mcalB_r'})}{\frac{\di \mbbP_{D'}^{g_{\bfb}}}{\di \mbbP_{D'}^{g_\bfb, \xi}}(\eta)}.
    \]
    The fact that the term in the expectation is well-defined and its precise value are provided by the following proposition:

    \begin{proposition}\label{prop:computderiv}
    For all fixed $\xi \in H^{-1}(\mcalB_r')$, it holds that $\di \mbbP_{D'}^{g_{\bfb}} \ll \di \mbbP_{D'}^{g_\bfb, \xi}$, and we have explicitly that
    \begin{equation}\label{eq:computderiv}
        \frac{\di \mbbP_{D'}^{g_{\bfb}}}{\di \mbbP_{D'}^{g_\bfb, \xi}}(\eta) = \ind{\eta \in \mcalA_\bfb, \eta \subset D'} \exp(\inv{2} (\xi, (H_{\partial D'} - H_{\partial (D' \setminus \eta)})_{\mcalB_r'}^{\mcalB_r'} \xi)_{\mcalB_r'} -2\lam ((H_{\partial D'})_{\mcalB_\bfb^-}^{\mcalB_r'} 1_{\mcalB_\bfb^-}, \xi)_{\mcalB_r'}).
    \end{equation}
\end{proposition}

We first pursue the proof of Theorem~\ref{thm:noncrossing} with this proposition in hand, and will turn to its proof just after. The computation of the expectation now relies on applying Lemma~\ref{lem:dub} to Equation~\eqref{eq:computderiv}. The determinant term reads\footnote{In this proof all occurrences of the Green's function $G_{D\setminus \eta}$ should be understood as the operator $(G_{D\setminus \eta})_{\mcalB_r'}^{\mcalB_r'}$ on $L^2(\mcalB_r')$, similarly for $G_{D'\setminus \eta}$ and $G_D$: we drop the subscripts and superscripts to lighten the notation.}
\begin{align*}
(\det_F(I_{\mcalB_r'}&-(G_{D \setminus \eta})^{1/2}(H_{\partial D'} - H_{\partial (D' \setminus \eta)})_{\mcalB_r'}^{\mcalB_r'} (G_{D \setminus \eta})^{1/2}))^{-1/2} \\
&= (\det_F(I_{\mcalB_r'}-(H_{\partial D'} - H_{\partial (D' \setminus \eta)})_{\mcalB_r'}^{\mcalB_r'} G_{D \setminus \eta}))^{-1/2} \\
&= \exp(\inv{2} m_D(\eta, \mcalB_r'))
\end{align*}
by Proposition~\ref{prop:detloopsoup} (note that this proposition implicitly contains the statement that the determinant is well-defined, and hence that the use of Lemma~\ref{lem:dub} is legitimate).

On the other hand, the term corresponding to the exponential in Lemma~\ref{lem:dub} can be rewritten as
\begin{align}
&\exp(\inv{2}\norm{2 \lam(I_{\mcalB_r'} - G_{D \setminus \eta}^{1/2}(H_{\partial D'} - H_{\partial (D' \setminus \eta)})_{\mcalB_r'}^{\mcalB_r'}G_{D' \setminus \eta}^{1/2})^{-1/2} G_{D \setminus \eta}^{1/2} (H_{\partial D'})_{\mcalB_\bfb^-}^{\mcalB_r'} 1_{\mcalB_\bfb^-}}^2) \nonumber\\
=& \exp(2 \lam^2 ((H_{\partial D'})_{\mcalB_\bfb^-}^{\mcalB_r'} 1_{\mcalB_\bfb^-},  (I_{\mcalB_r'} - (H_{\partial D'} - H_{\partial (D' \setminus \eta)})_{\mcalB_r'}^{\mcalB_r'}G_{D \setminus \eta})^{-1}G_{D \setminus \eta}  (H_{\partial D'})_{\mcalB_\bfb^-}^{\mcalB_r'} 1_{\mcalB_\bfb^-})_{\mcalB_r'} ) \nonumber\\
=& \exp(2 \lam^2 ((H_{\partial D'})_{\mcalB_\bfb^-}^{\mcalB_r'} 1_{\mcalB_\bfb^-}, G_{D}(H_{\partial D'})_{\mcalB_\bfb^-}^{\mcalB_r'} 1_{\mcalB_\bfb^-})_{\mcalB_r'} ) \nonumber\\
=& \exp(2 \lam^2 (1_{\mcalB_\bfb^-}, (H_{\partial D} - H_{\partial D'})_{\mcalB_\bfb^-}^{\mcalB_\bfb^-} 1_{\mcalB_\bfb^-})_{\mcalB_\bfb^-})\label{eq:computtermexplemma4},
\end{align}
where we used the identities in Lemma~\ref{lem:idloops}. This can finally be re-written as a difference of Dirichlet energies, as in Section~\ref{subsec:dirichlet}, so that we finally obtain 
\[
\frac{\di \mbbP_{D'}^{g_{\bfb}}}{\di \mbbP_{D}^{g_{\bfb}}} = \ind{\eta \in \mcalA_{\bfb}, \eta \subset D'} \exp(\frac{1}{2}m_{D}(\eta, D \setminus D')) \exp(-\inv{2} (\norm{g_{D, \bfb}}_{\nabla(D), \reg}^2 - \norm{g_{D', \bfb}}_{\nabla(D'), \reg}^2)).
\]
The exponential term in the equation above corresponds to $Z_{D, \bfb}/Z_{D''}$, which concludes the proof.
\end{proof}
We now turn to the proof of Proposition~\ref{prop:computderiv}.
\begin{proof}[Proof of Proposition~\ref{prop:computderiv}]
Consider $D'_\eps := \inset{x \in D': \dist(x,\mcalB_r') > \eps}$, $\mcalB'_\eps := \partial D'_\eps \setminus \partial D'$ and define $u_\xi^\eps$ as the unique harmonic function continuous on $\overline{D'}$ and harmonic on $D' \setminus \mcalB'_\eps$ with boundary values
\[
u_\xi^\eps = \begin{cases} u_\xi & \text{on }\mcalB'_\eps \\ 0 & \text{on } \partial D', \end{cases}
\]
where $u_\xi$ is the harmonic extension of $\xi$ to $D'$. Note that $u_\xi^\eps$ and $u_\xi$ coincide in $D'_\eps$. Let $\mbbQ$ be the law of $\Phi:=\Phi_{D'}^0+g_\bfb$ and $\tilde \mbbQ_\eps$ defined as
\begin{equation}\label{eq:rnderiveps}
\frac{\di \tilde \mbbQ_\eps}{\di \mathbb{Q}} = \exp((\Phi,u_\xi^\eps)_{\nabla(D')} - \inv{2}(u_\xi^\eps, u_\xi^\eps)_{\nabla(D')})
\end{equation}
so that $\tilde \Phi := \Phi - u_\xi^\eps$ is a zero-boundary GFF in $D'$ under $\tilde \mbbQ_\eps$, by Girsanov. Let $\eta, \tilde \eta$ be the level-lines for $\Phi, \tilde \Phi + u_\xi$ respectively.

Like in \cite[Proposition 3.3]{aru_extremal_2022}, we argue that the events $\inset{\dist(\eta, \mcalB_r')>\eps}, \inset{\dist(\tilde \eta, \mcalB_r')>\eps}$ are ($\tilde\mbbQ_\eps$-)a.s. equal, and that on these events $\eta$ and $\tilde \eta$ coincide. Whereas $(\eta, \Phi_\eta)$ is a local set under $\mbbQ$ by definition, $(\eta_\tau, \Phi_{\eta_\tau})$ is also a local set \cite[Lemma 3.12]{schramm_contour_2009}, where $\tau$ is the exit time of $\eta$ of $D_\eps'$, and $\Phi_{\eta_\tau}$ is the harmonic function in $D' \setminus \eta_\tau$ with boundary values given by $-\lam$ (resp. $+\lam$) on the left (resp. right) boundary of $\eta_\tau$, and coinciding with $g_\bfb$ on the other parts of $\partial D'$. Now by Lemma~\ref{lem:rnderivativecle4}, $(\eta_\tau, \tilde \Phi_{\eta_\tau})$ is also a local set under $\tilde \mbbQ_\eps$, with $\tilde \Phi_{\eta_\tau}= \Phi_{\eta_\tau} - (u_\xi^\eps)_{\eta_\tau}$. On the other hand, $(\tilde \eta_\tau, \tilde \Phi_{\eta_\tau})$, the level-line for $\tilde \Phi + u_\xi$, is also a local set under $\tilde \mbbQ_\eps$. Furthermore, the fields $\Phi$ and $\tilde \Phi + u_\xi$ coincide on $D_\eps'$. It follows from local uniqueness of level-lines (\cite[Lemma 15]{aru_bounded-type_2019}) that $\eta_\tau=\tilde\eta_\tau$ $\tilde\mbbQ_\eps$-almost surely. If $\tau=+\infty$, i.e. $\eta \subset D_\eps'$, then it is clear that $\tilde \eta \subset D_\eps'$ and $\eta=\tilde \eta$.

This implies that the law $\mbbP_{D'}^{g_\bfb, \xi}$ of the level line of $\Phi + u_\xi$ under $\tilde\mbbQ_\eps$, on the event that it stays at a distance $>\eps$ from $\mcalB_r'$, is absolutely continuous w.r.t. the law $\mbbP_{D'}$ of the level line of $\Phi$ under $\mbbQ$ on the event that it too stays at a distance $>\eps$ from $\mcalB_r'$, and the Radon-Nikodym derivative is given the conditional expectation of the RHS in Equation~\eqref{eq:rnderiveps} given $\eta$, which by Lemma~\ref{lem:rnderivativecle4} is given by:
\begin{equation}\label{eq:rnderivproposition}
\frac{\di \mbbP_{D'}^{g_\bfb}}{\di \mbbP_{D'}^{g_\bfb,\xi}}(\eta) = \exp(-(\Phi_\eta,u_\xi^\eps)_{\nabla(D')} + \inv{2}((u_\xi^\eps)_\eta, (u_\xi^\eps)_\eta)_{\nabla(D')}).
\end{equation}
We stress here that Lemma~\ref{lem:rnderivativecle4} above applies to a zero-boundary GFF, and a local set with $\Phi_\eta = h_\eta = 0$ on the boundary. In our framework, that amounts to considering the curve $\eta$ coupled with $\Phi - g_\bfb=\Phi_{D'}^0$, and where the harmonic part $h_\eta$ is the difference of the two-valued function that is equal to $-\lam$ on the left of $\eta$ and $+\lam$ on the right, and of the harmonic function $g_\bfb$.
We compute this expression on the event that $\dist(\eta, \mcalB_r')>\eps$, and take $\eps$ to $0$. Since $u_\xi$ is harmonic in $D'$, we can write
\begin{align}
(\Phi_\eta,u_\xi^\eps)_{\nabla(D')} &= (\Phi_\eta,- \Delta u_\xi^\eps)_{D'} + \int_{\partial D'} \Phi_{\eta} \cdot  \partial_\nu u_\xi^\eps \nonumber\\
&= (\Phi_\eta,- \Delta (u_\xi^\eps- u_\xi))_{D'} + \int_{\partial D'} h_{\eta} \cdot \partial_\nu u_\xi^\eps \nonumber\\
&= -\int_{\partial D'} \partial_\nu (u_\xi^\eps - u_\xi) \cdot h_{\eta} + \int_{\partial D'} (u_\xi^\eps - u_\xi) \cdot \partial_\nu h_\eta + \int_{\partial D'} h_{\eta} \cdot \partial_\nu u_\xi^\eps \nonumber\\
&= - (\xi, \partial_\nu h_\eta)_{\mcalB_r'} \label{eq:somerandomeq}
\end{align}
using that $h_\eta$ is harmonic and zero  on $\partial D'$, and $u_\xi^\eps - u_\xi$ vanishes on $\partial (D' \setminus D'_\eps)$. Note again the abuse of notation: the above denotes the distributional pairing between $\xi$ and $\partial_\nu h_\eta$. The explicit expression for $h_\eta$ (or for $g_\bfb$) in terms of the Poisson kernel allows to further write
\[
(\Phi_\eta,u_\xi^\eps)_{\nabla(D')} = \left(\xi, 2 \lam \int_{\mcalB_\bfb^-} H_{\partial D'}(\di y,\cdot)\right)_{\mcalB_r'} = 2\lam (\xi, (H_{\partial D'})_{\mcalB_\bfb^-}^{\mcalB_r'} 1_{\mcalB_\bfb^-})_{\mcalB_r'}.
\]
On the other hand, we claim that
\begin{equation}\label{eq:dirichletenergyxi}
((u_\xi^\eps)_\eta, (u_\xi^\eps)_\eta)_{\nabla(D')} = (\xi, (H_{\partial D'} - H_{\partial (D' \setminus \eta)})_{\mcalB_r'}^{\mcalB_r'} \xi)_{\mcalB_r'}
\end{equation}
Notice that this expression makes sense as $(H_{\partial D'} - H_{\partial (D' \setminus \eta)})_{\mcalB_r'}^{\mcalB_r'}$ is a smoothing operator, and that it suffices to prove it for $\xi$ smooth as for a general $\xi \in H^{-1}(\mcalB_r')$, taking $(\xi_n)_{n \geq 0} \subset C_0^\infty(\mcalB_r')$ converging to $\xi$ in $H^{-1}(\mcalB_r')$, we can write Equation~\eqref{eq:dirichletenergyxi} for $\xi_n$ and take the limit $n \to \infty$. This follows from $H_{\partial D'} - H_{\partial (D' \setminus \eta)}$ being continuous on $H^{-1}(\mcalB_r')$ and by $(u_{\xi_n}^\eps)_\eta$ converging to $(u_{\xi}^\eps)_\eta$ in $H_0^1(D')$.
Let $v_\eta$ be defined as
\begin{equation}\label{eq:defveta}
v_\eta(z) = (\xi, H_{D' \setminus \eta}(z, \cdot))_{\mcalB_r'}, \quad z \in D'\setminus\eta.
\end{equation}
On $\inset{\dist(\eta, \mcalB')>\eps}$, it holds that $(u_\xi^\eps)_\eta + v_\eta = u_\xi$, and $((u_\xi^\eps)_\eta, u_\xi)_{\nabla(D')} =0$ as $(u_\xi^\eps)_\eta$ vanishes on $\partial D'$ and $u_\xi$ is harmonic in $D'$. We can readily write:
\[
((u_\xi^\eps)_\eta, (u_\xi^\eps)_\eta)_{\nabla(D')}
= (v_\eta, v_\eta)_{\nabla(D')} - (u_\xi, u_\xi)_{\nabla(D')}.
\]
By harmonicity of $v_\eta$ and $u_\xi$, we have that
\begin{equation}
(v_\eta, v_\eta)_{\nabla(D')} - (u_\xi, u_\xi)_{\nabla(D')} = \int_{\mcalB_r'} (u_\xi \cdot \partial_\nu u_\xi - v_\eta \cdot \partial_\nu v_\eta).
\end{equation}
Using now the explicit expressions for $v_\eta$ and $u_\xi$ we obtain Equation~\eqref{eq:dirichletenergyxi}.

We conclude the proof of the proposition by taking $\eps \to 0$, as the terms computed above do not depend on $\eps$.
\end{proof}

\subsubsection{Further remarks}\label{sssec:cornoncrossing4}

In general, we know no closed form expression for the probability terms $\mbbP_D^{g_\bfb}(\mcalA_\bfb)$. However, in the case of the annulus such an expression can be found via stochastic calculus. Let us explain its origins in the following remark. 
\begin{remark}
    In \cite{izyurov_hadamards_2013, hagendorf_gaussian_2010}, $\sle_4$ is studied in the annulus. For a GFF in $(\ann{p},e^{ix},e^{iy})$, where $x, y \in \R$, with boundary conditions $-\lam$ on $(e^{iy}e^{ix})$, $+\lam$ on $(e^{ix}e^{iy})$ and $\mu$ on $C_{e^{-p}} \subset \partial \ann{p}$, it is shown that the corresponding level line is driven according to the annulus Loewner equation (see \cite{zhan_stochastic_2004}) with a driving process $\exp(iW_t)$, where the relative position $g_t(e^{ix})\exp(-iW_t)=:\exp(i B_t)$ is such that $B_t$ is a Brownian bridge from (the representative in $[0,2\pi)$ modulo $2 \pi$ of) $y-x$ to $\pi(1-\mu/\lam)$, of time-duration $p$, and up to its exit time $T$ of $[0,2\pi]$. The event that $T=p$ (the Brownian bridge stays within the interval) corresponds to the resulting level-line touching the inner boundary component $C_{e^{-p}}$, with the conformal modulus of the slit domain plays the role of the time parametrisation, whereas $T< p$ corresponds to the level-line reaching $e^{iy}$ from the left or right of $C_{e^{-p}}$, when $B_T=0$ or $2 \pi$ respectively. Observe that we have $T< p$ almost surely when $\abs{\mu}\geq\lambda$, as mentioned in Section~\ref{subsec:coupling}.
    In our case, the probabilities $\mbbP_{\ann{p}}^{g_{-\lam}}(\mcalA_{-\lam})$ and $\mbbP_{\ann{p}}^{g_{+\lam}}(\mcalA_{+\lam})$ correspond to the probability that a Brownian bridge from $y-x$ to $2 \pi$ (resp. to $0$) exits $[0,2\pi]$ through $2 \pi$ (resp. $0$), and the conformal modulus of the component of $\ann{p}\setminus \eta_T$ that is doubly connected is equal to $T$.
\end{remark}

Second, let us explain why the proof implies that the conformal restriction $\sle_4$ cannot be seen as a level line in the usual sense where it is considered to be a local set.

\begin{remark}\label{rem:levelline}
    The results above show that a.s. each instance of the conformal restriction $\sle_4$ can be seen as a sort of contour line of a single GFF: however, notice that in the construction the boundary conditions themselves depend on this contour line. As a matter of fact, a consequence of the proof is that there is no way to realise SLE as a generalised level-line of a single GFF even with random boundary conditions, as long as we do not permit the boundary conditions to depend on the lines ---in particular one cannot see the conformal restriction $\sle_4$ as a local set with respect to which the GFF would admit the usual markovian decomposition. Whereas we do not provide a proof of this claim here, let us outline an intuitive argument showing that coupling a GFF with conformal restriction $\sle_4$ as a level line entails that the boundary conditions of the GFF are a function of the topological class of $\sle_4$.
    
    Assume that $D=\ann{p}$ for simplicity and concreteness and let $D' \subseteq D$ be a test domain with $\dist(D', C_{e^{-p}})>\eps$ and $C_{e^{-p}}$ belonging to the connected component of $\ann{p} \setminus D'$ containing $1$. If a coupling $(\tilde \Phi_{\ann{p}}, \eta)$ of the sort were to exist, we could condition on $\mcalA_{+\lam}$, the event that $\eta$ passes to the left of $C_{e^{-p}}$. On $\ann{p - \de}$, $\de$ s.t. $e^{-p+\de}=e^{-p}+\eps$, we can write $\tilde \Phi_{\ann{p}} = \Phi_{\ann{p-\de}}^0 + u_{\xi}$ where $u_{\xi}$ is the harmonic extension of $\xi = \tilde \Phi_{\ann{p}}\restr{\partial \ann{p-\de}}$. The restriction property would then require that the level-line of $\tilde \Phi_{\ann{p}}$ conditioned to stay in $D'$ is absolutely continuous w.r.t. the level-line of $\Phi_{D'}^{g_{D'}} := \Phi_{D'}^0 + g_{D'}$ with the right Radon-Nikodym derivative, and the proof shows that this is the case provided that the expectation of Equation~\eqref{eq:computderiv} w.r.t. $\xi$ conditional on $\eta$ corresponds to $\exp(-m_{\ann{p}}(\eta, D')/2)$. It is reasonable to imagine that this property can hold for all $D'$ as above only if $\xi \sim \mcalN(g_{D, +\lam}, G_D)$ on $C_{e^{-p}+\eps}$, so that taking $\eps \to 0$ we recover that on the event $\mcalA_{+\lam}$, $\tilde \Phi_{\ann{p}}\restr{C_{e^{-p}}}=\Phi_D^{g_{D,+\lam}}\restr{C_{e^{-p}}}=+\lam$, i.e. that the boundary condition of $\tilde \Phi_D$ on $C_{e^{-p}}$ has to be $+\lambda$.
\end{remark}
\subsection{Constructing the crossing \texorpdfstring{$\sle_4$}{SLE4}}\label{subsec:crossingsle4}

We now go over the construction of the $\sle_4$ measure in the crossing case, i.e. corresponding to a measure on simple curves joining two points $x,y$ belonging to two different connected components of the boundary of a multiply connected domain $D$. There are two notable differences with the non-crossing case:
\begin{itemize}
    \item The relevant level-lines are coupled with multivalued GFFs, corresponding to the sum of a (Dirichlet, zero-boundary) GFF and a deterministic multivalued harmonic function with locally constant boundary conditions with a change of $\pm 2 \lam$ around $x$ and $y$ on $\partial D$, and a change of branch describing the winding of the curve along the different boundary components. In general these level-lines are again conditioned on topological events, but in the case of doubly connected domains the coupling can be realised directly with a single compactified GFF--a sum of a (Dirichlet, zero-boundary) GFF and multivalued harmonic function with random monodromy.
    \item The topological events $\mcalA_\bfb, \bfb=(b_1, \cdots, b_n) \in \inset{\pm \lam}^n$, that amounted to assignations of $-\lam$ or $+\lam$ as boundary conditions of a GFF to the connected components of the boundary of the domain, correspond in the crossing case to events $\mcalA_\bfb, \bfb = (\sbin, b_1, \dots, b_n) \in \Lambda^{n+1} := \Z \times (2\lam \Z +\lam)^n$. The component $\sbin$ describes the winding of any $\eta \in \mcalA_\bfb$ around $\bin$, while the other coordinates represent the change in branch of the multivalued harmonic function acting as the mean of the GFF when approaching the other boundary component.
\end{itemize}
We will follow a similar structure - we start with the set-up and statement, then give the proof and finish with some corollaries and remarks.

\subsubsection{Set-up and statement}

Let now $(D,x,y)$ be a crossing admissible domain, for which we may assume that $\bin=C_{e^{-p}}$, $\bout = C_1 = \partial \D$, that all $\mcalB_j, 1 \leq j \leq n$ are given by circles of centre $z_j$, and that $x=1$. Denote by $U$ its preimage by $q: z \mapsto e^{iz}$, a subset of the strip $S_p := \inset{a+ib: a \in \R, 0<b<p}$. For $\bfb \in \Lambda^{n+1}$, let $w$ be such that $e^{iw}=y$ and $0 \leq w - 2 \pi \sbin <2\pi$. We define $g_\bfb$ as the multivalued harmonic function in $D$ whose lift $\tilde g_\bfb$ to $S_p$ is the harmonic extension of the following boundary conditions: $\lam$ on the segments $[0, 2\pi], w+[0, 2\pi]$, $b_j$ on the lift of $\mcalB_j$ that contains the point $\tilde z_j$ in the fiber of $z_j$ such that $0 \leq \Re \tilde z_j < 2 \pi$, and with the values on the other parts of the boundary fixed by the condition $\tilde g_\bfb(z+2k \pi) = \tilde g_\bfb(z) + 2k\lam$.

We then let $\mcalA_\bfb$ denote the set of simple curves $\eta:x \to y$ such that the parts of $\partial U$ that are contained between any two adjacent lifts of $\eta$ correspond to constant boundary values of $g_\bfb$. We finally define $\mbbP_D^{g_\bfb}$ as the law of the generalised level line of a GFF $\Phi_D$ in $D$ with boundary conditions given by $g_\bfb$ (see also Appendix~\ref{app:existencelevelline}).

\begin{figure}[!ht]
	   \centering
	   \includegraphics[scale=0.53]{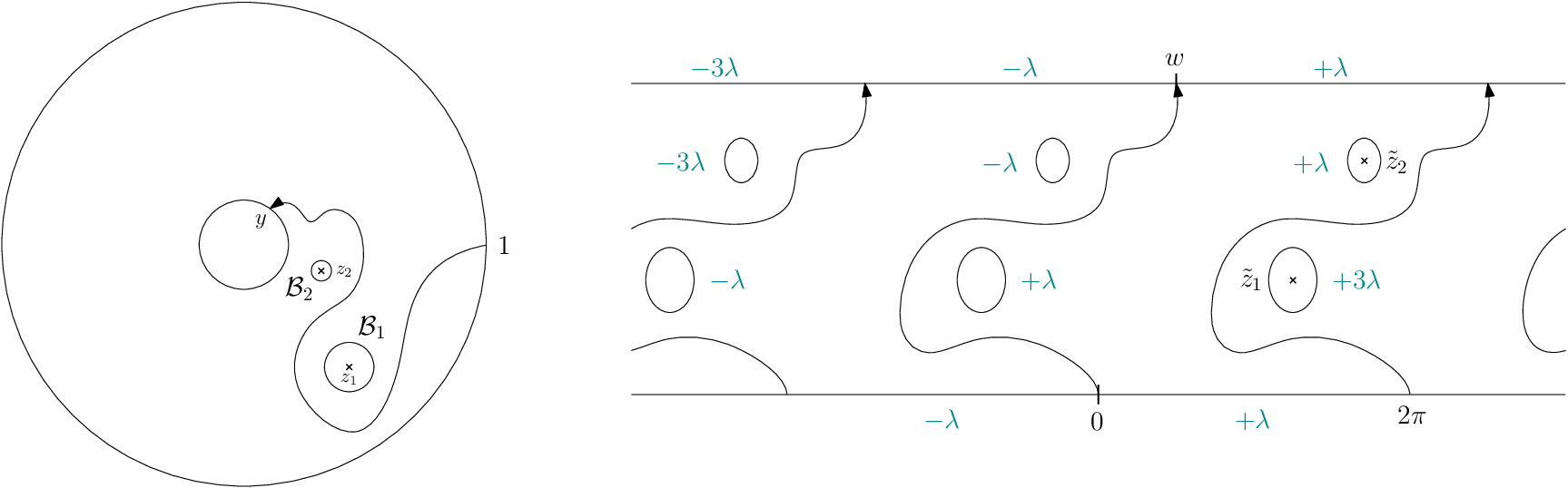}
        \caption{A configuration with a curve in $\mcalA_{0,+3\lam,+\lam}$.}
        \label{fig:crossing_disk}
\end{figure}

The point of the definition is just to guarantee that $\tilde g_\bfb$ corresponds to the boundary conditions $\lambda(1+2k)$ on every part of the boundary of $U$ between the lifts of $\eta$ starting from $2 k \pi$ and $2 (k+1) \pi$ (for any $\eta \in \mcalA_\bfb$: see Figure~\ref{fig:crossing_disk}).

The Dirichlet energy $\normreg{g_\bfb}$ is then defined as $\norm{\tilde g_\bfb}_{\nabla(D'), \reg}$, where $D' \subset U$ is any fundamental domain for the action of the translation by $2\pi$ in $U$. For instance, one may take $D'$ as the intersection of $U$ and a parallelogram with vertices $-\pi,\pi, w-\pi, w+\pi$, for any $w$ such that $e^{iw}=y$, and regularise the divergences at $0$ and $w$ like in Section~\ref{subsec:dirichlet}.

\begin{theorem}\label{thm:crossing}
    Let $(D,x,y)$ be a crossing admissible domain. The measure $\nu_D$ on simple curves from $x$ to $y$ in $D$ defined as
    \begin{equation}\label{eq:defcrossingsle4measure}
        \nu_D := \sum_{\bfb \in \Lambda^{n+1}} \ind{\eta \in \mcalA_\bfb} \mbbP_D^{g_\bfb} \cdot Z_{D,\bfb},
    \end{equation}
    where
    \[
        Z_{D,\bfb} = \exp(-\frac{1}{2} \normreg{g_\bfb}^2)
    \]
    is the measure of $\sle_4$ in $(D,x,y)$.
\end{theorem}

\begin{corollary}\label{cor:partitionfunctioncrossing4}
    The partition function of $\sle_4$ in a crossing admissible domain $(D,x,y)$ can be expressed as 
    \[
        Z_D(x,y) := \sum_{\bfb \in \Lambda^{n+1}} \mbbP_D^{\bfb}(\mcalA_\bfb) \exp(-\inv{2} \normreg{g_\bfb}^2)
    \]
    and it is finite.
\end{corollary}

The summability of the series is provided by the exponential decay of $Z_{D, \bfb}$ as $\norm{\bfb}_\infty \xrightarrow[]{} \infty$. Indeed, one can write
\[
g_\bfb = g_{\inset{0,\lam,\cdots,\lam}} + 2 \lam\sbin \cdot v_{\sbin} + \sum_{j=1}^n (b_j-\lam)  v_j,
\]
where $v_j$ for $1 \leq i \leq n$ (resp. $v_{\sbin}$) is a harmonic function whose boundary values are one on $\mcalB_j$ (resp. on $\bin$), zero elsewhere. When an index $b_j, j \in \inset{\text{in}, 1, \dots, n}$ goes to $\pm\infty$, the leading term in $\normreg{g_\bfb}^2$ is given by $b_j^2 \norm{v_j}_{\nabla(D)}$ and therefore $\exp(-\inv{2}\normreg{g_\bfb}^2)$ decays exponentially fast.

\subsubsection{Proof of the theorem}

\begin{proof}[Proof of Theorem~\ref{thm:crossing} ]
Let $\tilde \nu_D$ be the measure in the statement of the theorem, $D'' \subset D$ be a test domain, and let $\bfb'\in \Lambda^{n+1}$ be such that $\inset{\eta \in D''}\subset\mcalA_{\bfb'}$. This time again, instead of comparing $\tilde \nu_D$ directly to $\nu_{D''}(x,y)$, we use an intermediary domain $D'=D \setminus \delta$, where $\delta$ is the arc of $\partial D'' \cap D$ from the right of $x$ to $\bin$. In particular, $\sle_4$ from $x$ to $y$ in $D'$ is the non-crossing $\sle$ that was previously constructed ($D'$ is conformally equivalent to a non-crossing admissible domain), and its law $\nu_{D'}(x,y)$ is described by Theorem~\ref{thm:noncrossing}, in which case it holds that
\[
    \frac{\di \nu_{D''}(x,y)}{\di \nu_{D'}(x,y)}(\eta)= \ind{\eta \subset D''} \exp(\inv{2} m_{D'}(\eta, D'\setminus D'')).
\]
The set of $\bfb \in \Lambda^{n+1}$ such that $\inset{\eta \in D'} \cap \mcalA_\bfb \neq \emptyset$ can be naturally identified with $\inset{\pm \lam}^n$, as ${\eta \in D'}$ fixes the winding around $\bin$, and leaves the freedom of going to the left or the right of each $\mcalB_j$ in $D'$. Let us under this identification consider $\bfb \in \inset{\pm \lam}^n$, and denote $g_{D'} = g_{D',\bfb}, g_D = g_{D, \inset{\sbin, \bfb}}$ (where $\sbin$ is determined by $D'$) for simpler notation. We have to prove that $\mbbP_{D'}^{g_{D'}}$ is absolutely continuous with respect to $\mbbP_D^{g_D}$ with Radon-Nikodym derivative given by
\[
\frac{\di \mbbP_{D'}^{g_{D'}}}{\di \mbbP_D^{g_D}}(\eta) = \ind{\eta \subset D'}\ind{\eta \in \mcalA_\bfb}\frac{Z_{D,\bfb}}{Z_{D',\bfb}} \exp(\inv{2} m_D(\eta, D \setminus D')).
\]
This is achieved once again by computing the average of the Radon-Nikodym derivative for fixed $\xi$ over $\xi \sim \mcalN(0,G_{D \setminus \eta}\restr{\de})$. Notice that conditionally on $\eta$, $\Phi_D^{g_D}\restr{D'}$ can be written $\Phi_{D \setminus \eta}^0 + u_\xi + h_\eta$ where $u_\xi$ is the harmonic extension of $\xi$ to $D'$, $h_\eta$ is multivalued with every branch in $D'$ of the form $(2k-1)\lam + 2\lam \ind{\text{right of } \eta}(z)$, so that the law $\mbbP_D^{g_D}$ knowing $\xi$ and restricted to the event $\mcalA_\bfb \cap \inset{\eta \subset D'}$ is the same as $\mbbP_{D'}^{g_{D'}, \xi}$ from the non-crossing case thanks to the construction of $g_D$ and the definition of $\mcalA_\bfb$. Similarly to the non-crossing case, we start with the computation of the Radon-Nikodym derivative for $\xi$ fixed.\footnote{Although $\de$ is one single crosscut, it is part of the boundary of both $D_l'$ and $D_r'$, the left and right connected components of $D'$ relative to $\eta$: as such, we denote formally by $\mcalB_l'$ (resp. $\mcalB_r'$) $\de$ seen as the boundary of $D_l'$ (resp. $D_r'$). We will be brought to consider Brownian excursions from $\de$ to $\de$ in $D$, which we understand as excursions between the two preimages of $\de$ by $x \mapsto e^{ix}$, that we identify to $\mcalB_l'$ and $\mcalB_r'$. Denote also by $\nu_l$ (resp. $\nu_r$) the outwards normal vector to $D_l'$ (resp $D_r'$) at $x \in \de$: we abuse notation and denote $\partial_{\nu_l}g_{D'}$ for the normal derivative of $g_{D'}$ on $\partial D_l'$, and stress that $\partial_{\nu_l}g_{D'} \neq -\partial_{\nu_r}g_{D'}$ on $\de$.\nopagebreak}

\begin{proposition}\label{prop:computderivcrossing}
    For all $\xi \in H^{-1}(\delta)$, it holds that $\di \mbbP_{D'}^{g_{D'}} \ll \di \mbbP_{D'}^{g_{D'}, \xi}$, and we have explicitly that
    \begin{equation}\label{eq:computderivcrossing}
        \frac{\di \mbbP_{D'}^{g_{D'}}}{\di \mbbP_{D'}^{g_{D'}, \xi}}(\eta)\ind{\eta \in \mcalA_\bfb, \eta \subset D'} = \ind{\eta \in \mcalA_\bfb, \eta \subset D'} \exp(\inv{2} (\xi, M \xi)_{\de} - ((\partial_{\nu_l} + \partial_{\nu_r}) g_{D'}, \xi)_{\de}),
    \end{equation}
    where
    \[
        M=(H_{\partial D'} - H_{\partial (D' \setminus \eta)})_{\mcalB_l'}^{\mcalB_l'} + (H_{\partial D'} - H_{\partial (D' \setminus \eta)})_{\mcalB_r'}^{\mcalB_r'} + 2 (H_{\partial D'})_{\mcalB_l'}^{\mcalB_r'}
    \]
    with the subscripts/superscripts $\mcalB_l', \mcalB_r'$ to differentiate when the operators correspond to Brownian excursions from $\mcalB_j'$ to $\mcalB_k'$, $j,k=l,r$.
\end{proposition}

We first finish the proof of Theorem~\ref{thm:crossing} before going over the proof of Proposition~\ref{prop:computderivcrossing}. When computing the average of Equation~\eqref{eq:computderivcrossing} with Lemma~\ref{lem:dub}, the determinant term is that of\footnote{In this proof all occurrences of the Green's function $G_{D\setminus \eta}$ should be understood as the operator $(G_{D\setminus \eta})_{\de}^{\de}$ on $L^2(\de)$, similarly for $G_{D\setminus \eta}$: we drop the subscripts and superscripts to lighten the notation.}
\begin{equation}
I - G_{D \setminus \eta}(H_{\partial D'} - H_{\partial (D' \setminus \eta)})_{\mcalB_l'}^{\mcalB_l'} - G_{D \setminus \eta}(H_{\partial D'} - H_{\partial (D' \setminus \eta)})_{\mcalB_r'}^{\mcalB_r'} - 2 G_{D \setminus \eta}(H_{\partial D'})_{\mcalB_l'}^{\mcalB_r'}.
\end{equation}
This can again be understood with another decomposition formula for Brownian loops. One can further decompose the expressions obtained after applying Lemma~\ref{lem:idloops} by distinguishing on which side the Brownian path leaves and hits $\de$:
\begin{align*}
G_D - G_{D \setminus \eta} = G_{D \setminus \eta} \paren{(H_{\partial D'} - H_{\partial (D' \setminus \eta)})_{\mcalB_l'}^{\mcalB_l'} - (H_{\partial D'} - H_{\partial (D' \setminus \eta)})_{\mcalB_r'}^{\mcalB_r'} - 2 (H_{\partial D'})_{\mcalB_l'}^{\mcalB_r'}} G_D.
\end{align*}
From here, Proposition~\ref{prop:detloopsoup} once again provides that the determinant term in the Gaussian integration is equal to
\[
m_{D}(\eta, D \setminus D').
\]
This also shows that
\[
G_{D \setminus \eta}^{1/2} (I - G_{D \setminus \eta}^{1/2}MG_{D \setminus \eta}^{1/2})^{-1} G_{D \setminus \eta}^{1/2} = G_D
\]
and the exponential term is equal to
\begin{equation}\label{eq:exponentialterm}
\exp(\inv{2} ((\partial_{\nu_l} + \partial_{\nu_r}) (g_{D'}), G_D ((\partial_{\nu_l} + \partial_{\nu_r}) g_{D'}))_{\de})=\exp(-\inv{2} ((\partial_{\nu_l} + \partial_{\nu_r}) g_{D'}, G_D ((\partial_{\nu_l} + \partial_{\nu_r}) (g_{D}-g_{D'})))_{\de}),
\end{equation}
given that the gradient of $g_D$ is single valued. 
Now we need to invert $(G_D)_{\de}^{\de}$, i.e. the operator
\[
f \mapsto \int_{\de} G_D(\cdot, x) f(x) \ldx
\]
that from a function in $L^2(\de)$ outputs a continuous function in $D$, harmonic in $D \setminus \de$, with values $0$ on $\partial D$ and whose normal derivative jumps by $f(x)$ across $\de$ at $x$. If one now replaces $g_{D'}$ with the multivalued function $g_{D',m}$ that is equal to $g_D + 2 n \lam$ in every branch in $D'$, then $g_D - g_{D',m}$ is single-valued and satisfies the conditions above\footnote{In fact the inverse of the Green's function here is also a Neumann-type operator that computes the difference of the normal derivatives of the harmonic extensions of $f$ on both sides of $\de$.}: it follows that $G_D ((\partial_{\nu_l} + \partial_{\nu_r}) (g_D - g_{D',m})) = g_D - g_{D',m}$. We rewrite Equation~\eqref{eq:exponentialterm} as
\[
\exp(-\inv{2}((g_D - g_{D',m}), (\partial_{\nu_l} + \partial_{\nu_r})g_{D'})_\de) = \exp(\inv{2}\int_{D'} \nabla(g_{D', m} -g_{D}) \cdot \nabla g_{D'})
\]
and to conclude it suffices to note that
\begin{align*}
\int_{D'} \nabla(g_{D', m} -g_{D}) \cdot \nabla g_{D'} = \int_{\partial D'} g_{D'} \cdot \partial_{\nu} (g_{D'} -g_{D}) = 4 \lam^2 \int_{\partial D_l'} \partial_{\nu} (g_{D'} -g_{D}) = \norm{\nabla g_D'}_{D', \reg}^2 - \norm{\nabla g_D}_{D', \reg}^2
\end{align*}
given that $g_D$ has monodromy $2\lam$ and $\partial_{\nu_l}g_D = - \partial_{\nu_r}g_D$ on $\de \subset \partial D_l'$.
\end{proof}

We conclude with the proof of Proposition~\ref{prop:computderivcrossing}.

\begin{proof}[Proof of Proposition~\ref{prop:computderivcrossing}]
    The proof starts identically to that of Proposition~\ref{prop:computderiv}, up to the computation of $((u_\xi^\eps)_\eta, (u_\xi^\eps)_\eta)_{\nabla(D')}$. Write 
    \[
        u_\xi(z) = \int_{\delta} H_{D \setminus \de}(z,\di y) \xi(y) = \underbrace{\int_{\mcalB_l'} H_{D'}(z,\di y) \xi(y)}_{=:u_l(z)} + \underbrace{\int_{\mcalB_r'} H_{D'}(z,\di y) \xi(y)}_{{=:u_r(z)}}
    \]
    We still have that $u_\xi = (u_\xi^\eps)_{\eta} + v_\eta$, with $v_\eta$ defined as
    \[
    v_\eta(z)= \int_{\mcalB_l'} H_{D_l'}(z,\di y) \xi(y) + \int_{\mcalB_r'} H_{D_r'}(z,\di y) \xi(y).
    \]
    Similarly to the proof of Proposition~\ref{prop:computderiv}, we assume that $\xi$ is smooth, and compute
    \begin{align*}
        ((u_\xi^\eps)_\eta, (u_\xi^\eps)_\eta)_{\nabla(D')} &= (v_\eta, v_\eta)_{\nabla(D')} - (u_l, u_l)_{\nabla(D')} - (u_r, u_r)_{\nabla(D')} + 2 (u_l, \partial_\nu u_r)_{\mcalB'} \\
        &= \int_{\mcalB_l'} \xi \cdot (H_{\partial D'}-H_{\partial D_l'})_{\mcalB_l'}^{\mcalB_l'} \xi + \int_{\mcalB_r'} \xi \cdot (H_{\partial D'}-H_{\partial D_r'})_{\mcalB_r'}^{\mcalB_r'} \xi + 2 \int_{\mcalB_l'} \xi \cdot (H_{\partial D'})_{\mcalB_l'}^{\mcalB_r'} \xi,
    \end{align*}
    seeing $\xi$ as an element of $H^{-1}(\mcalB_l')$ or $H^{-1}(\mcalB_r')$.

    The computation of the term $(\Phi_\eta, u_\xi^\eps)_{\nabla(D')}$ is identical to Equation~\eqref{eq:somerandomeq}, taking into account that $\mcalB_l'$ and $\mcalB_r'$ are seen as distinct parts of the boundary of $D'$ (they appear separately when applying the divergence theorem):
    \[
        (\Phi_\eta, u_\xi^\eps)_{\nabla(D')} = -\int_{\de} \xi \cdot \partial_\nu h_\eta =  -\int_{\mcalB_l'} \xi \cdot \partial_\nu h_\eta - \int_{\mcalB_r'} \xi \cdot \partial_\nu h_\eta = - (\partial_\nu g_{D'}, \xi)_{\mcalB_l'} - (\partial_\nu g_{D'}, \xi)_{\mcalB_r'}. \qedhere
    \]
\end{proof}

\subsubsection{Corollaries for annular domains}
%
    The conditioning on the events $\mcalA_\bfb$ corresponds loosely speaking to two requirements: that the level-line $\eta$ of $\mbbP_D^{g_\bfb}$ winds around $\bin$ (the boundary component containing the target point $y$) a number of times given by $\sbin$, and that for each $1 \leq j \leq n$ it winds around the other boundary components $\mcalB_j$ a number of times given by $b_j$. While the latter conditioning is necessary in general, the former turns out to be superfluous: indeed, the monodromy of $g_\bfb$ already fixes the winding of $\eta$ around $\bin$, in the sense that the level line of $\mbbP_D^{g_\bfb}$ almost surely satisfies the property that its winding around $\bin$ is given by $\sbin$. We refer to Appendix~\ref{app:existencelevelline} for the proof of this fact.
%
In particular, when $D$ is a conformal annulus (corresponding to $n=0$), there are no other boundary components than those containing the target points $x$ and $y$ and thus no conditioning in the construction of the $\sle_4$ measure. We hence get the following corollary:
\begin{corollary}
    When $D$ is a conformal annulus, we have that
    \[
        Z_D(x,y) = \sum_{\sbin \in \Z} \exp(-\inv{2} \normreg{g_{\sbin}}^2),
    \]
    and $\nu_D$ corresponds to the law of a level-line of a compactified GFF $\Phi=\Phi_D^0 + g$, where $g$ is a random multivalued harmonic function in $D$, independent of $\Phi_D^0$, taking values $g=g_{\sbin}$ with probability $Z_{D, \sbin}/\sum_{\sbin' \in \Z} Z_{D,\sbin'}$.
\end{corollary}
The law of the winding of the curve under $\nu_D$ can then be computed explicitly.
\begin{corollary}\label{cor:windingsle}
    For $\al$ the representative of $(\arg y - \arg x)/(2\pi)$ in $[0,1)$, it holds that
    \begin{equation}\label{eq:wind1}
        \nu_{\ann{p}}^\#(x,y)[\eta \in \mcalA_k]= \frac{Z_{D,k}}{Z_D} = \frac{e^{-\frac{\pi^2}{2p} \cdot (k+\al)^2}}{e^{-\frac{\pi^2}{2p}\al^2}\vartheta_3(i \frac{\pi}{2p} \al | i \frac{\pi}{2p})}, \quad k \in \Z,
    \end{equation}
    where $\vartheta_3(z | \tau) = \sum_{k\in \Z} e^{\pi i k^2 \tau}e^{2 \pi i k z}$ is a Jacobi theta function.
\end{corollary}
\begin{proof}
    The event $\mcalA_{\sbin}$ exactly corresponds to the event that the index of $\eta$ around $y$ equals $\sbin+\al$ (with the index defined consistently with $\al$), so that it suffices to compute $Z_{D, \sbin}$ to identify the desired probabilities. Notice that $g_{\sbin}=g_0 + \sbin f$, for $f=2 \lam \log \abs{z}/p$. Thus,
    \begin{align*}
        \normreg{g_{\sbin}}^2 &= \normreg{g_0}^2 + \sbin^2 \normh{f}^2 + 2 \sbin (g_0, f)_{\nabla(D)} \\
        &= \normreg{g_0}^2 + \sbin^2 \cdot \frac{\pi^2}{p} + 2 \sbin (g_0, f)_{\nabla(D)}
    \end{align*}
    The term $(g_0, f)_{\nabla(D)}$ is a function of $\al$, and the dependence is affine, given that $g_0$ can be constructed from the case $\al=0$ by precomposing by a rotation that is a function of the modulus of the input. Since at $\al=1$ it must be equal to $\pi^2/p$, we deduce that
    \[
        \normreg{g_{\sbin}}^2 = \frac{\pi^2}{p} \cdot (\sbin+\al)^2 + C. \qedhere
    \]
\end{proof}


\section{\texorpdfstring{$\slek$}{SLE kappa} as a mixture of boundaries of restriction measures and CLE}\label{sec:slek}

We now turn to a more general construction of $\slek$ in multiply connected domains for the whole range $\kappa \in (8/3,4]$. Here the constructions and proofs are simpler than in the previous section, but the philosophy of having to separate topological events remains the same. The constructions here are specific to the non-crossing case. 

We start by recalling classical facts about restriction measures and CLEs and then state and prove our construction.

\subsection{Preliminaries}

\subsubsection{Restriction measures and Poisson point processes of Brownian excursions}\label{subsec:restriction}
We first recall the definition of one-sided restriction measures, as introduced in \cite{lawler_conformal_2003}. For $D$ a bounded simply connected domain and distinct $x, y \in \partial D$, consider a probability measure $\mbbP$ on closed $R \subset \overline{D}$ such that $R \cap \partial D = (yx)$ and $D \setminus R$ is connected. We say that $\mbbP$ is a \emph{one-sided restriction measure} with exponent $\al$ if there exists $\al \geq 0$ such that for all closed $K \subset \overline{D}$ with $K \cap (yx)=\emptyset$ and $D \setminus K$ is simply connected\footnote{The $\sigma$-field we consider is generated by the events $R \cap K = \emptyset$, which is the same as the Borel $\sigma$-field generated by the Hausdorff metric on closed subsets $\overline D$, see \cite[Section 3]{lawler_conformal_2003}.},
\[
\mbbP(R \cap K = \emptyset)=(\Phi_K'(x)\Phi_K'(y))^\al=\paren{\frac{H_{\partial (D \setminus K)}(x,y)}{H_{\partial D}(x,y)}}^\al,
\]
where $\Phi_K: D\setminus K \to D$ is any conformal equivalence with $\Phi_K(x)=x, \Phi_K(y)=y$.
These restriction measures are conformally invariant and can be constructed in several ways (see for instance \cite{werner_conformal_2005} for a survey). We will use the construction involving Brownian excursions, which generalises nicely to the multiply-connected setting.
Consider $(D,x,y)$ a (non-crossing) admissible domain and a Poisson point process (PPP) of intensity given by $h$ (calling $h$ the \emph{intensity constant}) times the (infinite) measure of Brownian excursions from $\mcalB_l$ in $D$
\begin{equation}\label{eq:defbrownianexcursions}
\pi \iint_{\mcalB_l \times \mcalB_l} \mu_{\partial D}(z,w) \ld{z} \ld{w},
\end{equation}
where $\mu_{\partial D}(z,w) = H_{\partial D}(z,w) \mu_{\partial D}^\#(z,w)$ is the Brownian measure on excursions\footnote{We refer to \cite[Section 3.3]{BrownianLoopSoup} once again for the precise definition of these measures and the fact that they satisfy conformal invariance and the restriction property, which is all we need for this section. The restriction property states that $\mu_{\partial D}(z,w)$ restricted to curves staying in a subdomain $D' \subset D$ agreeing with $D$ near $z$ and $w$ coincides with $\mu_{\partial D'}(z,w)$.} from $x$ to $y$ in $D$.
Denote by $R_{h}^{D, \mcalB_l}$ its filling, i.e. the complement of the connected component of $D$ minus a sample of the PPP containing $\mcalB_r$.
In the case where $D$ is simply connected, $R_h^{D, \mcalB_l}$ gives rise to the one-sided restriction measure \cite{werner_conformal_2005} \footnote{The factor of $\pi$ in Equation~\eqref{eq:defbrownianexcursions} is due to our convention for the Poisson kernel, in order for $R_{h}^{D, \mcalB_l}$ to be a one-sided restriction measure of exponent $h$ as introduced in \cite{lawler_conformal_2003}. Indeed, when $D$ is simply connected, by Equation~\eqref{eq:dubproof}, Equation~\eqref{eq:restrictionequality} becomes
\[
\paren{\frac{H_{\partial D'}(x,y)}{H_{\partial D}(x,y)}}^h,
\]
i.e. $R_{h}^{D, \mcalB_l}$ is a one-sided restriction measure of exponent $h$. The identification of the relation between the intensity constant of the PPP and the value of the parameter of the corresponding restriction measure was also carried out in \cite[Theorem 3.7]{lupu_level_2024}.}. The key is the following calculation, which generalises to the multiply-connected setting and that we will make use of in the proof of the Theorem below.
\begin{lemma}\label{lem:excursionproba}
    In the setting above and for $D' \subset D$ such that $D \setminus D'$ is at positive distance from $\mcalB_l$, the probability that all loops in $R_{h}^{D, \mcalB_l}$ stay in $D'$ is equal to
    \begin{equation}\label{eq:restrictionequality}
        \exp(-\frac{h \pi}{4} (\norm{u_{D'}}_{\nabla(D'), \reg}^2 - \norm{u_D}_{\nabla(D), \reg}^2)),
    \end{equation}
    where $u_D$ (resp. $u_{D'}$) is harmonic with boundary values $-1$ on $\mcalB_l$, $+1$ on the rest of $\partial D$ (resp. $\partial D'$).
\end{lemma}
\begin{proof}
    By definition of the PPP and the restriction property satisfied by the excursion measures, the random variable that counts the loops in a realisation of the PPP that exit $D'$ is Poisson with parameter
    \[
        h \pi \iint_{\mcalB_l' \times \mcalB_l'} (H_{\partial D}(x,y) - H_{\partial D'}(x,y)) \ldx \ldy.
    \]
    As in Section~\ref{subsec:dirichlet}, this might be rewritten as
    \[
        \frac{h \pi}{4} (\norm{u_{D'}}_{\nabla(D'), \reg}^2 - \norm{u_D}_{\nabla(D), \reg}^2). \qedhere
    \]
\end{proof}
The same construction can be replicated when replacing $\mcalB_l$ with $\mcalB_\bfb^-$ for $\bfb \in \inset{\pm \lam}^n$.

\subsubsection{Conformal Loop Ensembles}\label{subsec:cle}
The Conformal Loop Ensembles (CLE) are another one-parameter family of conformally invariant measures, this time on collections of non-nested disjoint simple loops in a domain $D$ for the parameter range $8/3 < \kappa \leq 4$. There also admit various constructions (see \cite{sheffieldConformalLoopEnsembles2012}), but we will see them here as the collection of the outermost boundaries of clusters in a Brownian loop-soup, i.e. a Poisson point process with respect to the Brownian loop measure, with intensity $\al=\cc(\kappa)/2$, $0 < \al \leq 1/2$. We stress that the intensity parameter in \cite{sheffieldConformalLoopEnsembles2012, lawler_partition_2009} and the central charge $\cc$ differ by a factor of $1/2$: see also \cite{lupu_convergence_2018} for a more thorough explanation.

The CLE also satisfies a form of conformal restriction: for $D' \subset D$, consider the set $\Gamma'$ obtained from a realisation $\Gamma$ of the CLE in $D$ consisting of all the loops that intersect $D'$, and let $D_\Gamma'$ be the random set given by the union of $D'$ with the filling of $\Gamma'$. Then, conditionally on $D_\Gamma'$, the law of the loops of the CLE that do not intersect $D_\Gamma'$ is that of a CLE in $D \setminus D_\Gamma'$.

\subsection{Constructing the non-crossing \texorpdfstring{$\slek$}{SLE kappa}}

We work with the notation of the previous section, only redefining the parameters $\bfb$ to belong to $\inset{\pm 1}^n$ (the purpose of the factor $\lam$ was just to make a notational connection with the boundary conditions of the GFFs, but all that matters is the sign of $\bfb$). We will also consider for $D$ an admissible domain and $\bfb \in \inset{\pm 1}^n$ the harmonic function $u_\bfb$ given by the boundary values $-1$ on $\mcalB_\bfb^-$, $+1$ on $\mcalB_\bfb^+$, so that $u_\bfb = g_\bfb / \lam$ from the previous section.

Consider $R$ a sample of $R_{h}^{D,\mcalB_\bfb^-}$, $\Gamma$ a sample of an independent $\cle_{\kappa}^D$, and $C$ the union of $R$ with all the loops of $\Gamma$ that it intersects. We let $\mbbP_D^\bfb$ be the law of the right boundary of the component of $C$ that contains the arc $(yx)$, i.e. the boundary of its filling minus $(yx)$. This is a simple curve in $D$ (see \cite{werner_clekappa_2013}, just before Section 3, for an argument in the simply connected case: the general case is identical).
In other words, $\mbbP_D^\bfb= F_\ast(\cle_{\kappa}^D \otimes R_{h}^{D,\mcalB_\bfb^-})$, where $F$ identifies the path according to the construction above given a realisation of the $\cle$ and of the PPP of Brownian excursions.

\begin{theorem}\label{thm:noncrossingslekappa}
    Let $(D,x,y)$ be a non-crossing admissible domain and $8/3 < \kappa \leq 4$. The measure $\nu_D$ on simple curves from $x$ to $y$ in $D$ defined as
    \begin{equation}\label{eq:defnoncrossingslekappameasure}
        \nu_D := \sum_{\bfb \in \inset{\pm 1}^n} \ind{\eta \in \mcalA_\bfb}\mbbP_D^{\bfb} \cdot Z_{D,\bfb},
    \end{equation}
    where
    \begin{equation}\label{eq:defzk}
        Z_{D,\bfb} = \exp(- \frac{\pi h}{4} \normreg{u_\bfb}^2),
    \end{equation}
    is the measure of $\sle_\kappa$ in $(D,x,y)$.
\end{theorem}

From this construction we can infer the expression of the partition function for non-crossing $\sle_\kappa$.

\begin{corollary}\label{cor:partitionfunctionnoncrossingkappa}
    The partition function of $\sle_\kappa$ from $x$ to $y$ in $D$, $8/3<\ka\leq4$, can be expressed as
    \[
        Z_D(x,y):= \sum_{\bfb \in \inset{\pm 1}^n}\exp(-\frac{\pi h}{4} \normreg{u_\bfb}^2) \cdot \mbbP_D^\bfb(\mcalA_\bfb)
    \]
    and it is finite.
\end{corollary}

As a minor check we note that $\pi h_4/4=\pi/16 = \lam^2/2$ and hence the expression above agrees with Equation~\eqref{eq:defz4} when $\kappa=4$.

\subsubsection{Proof of Theorem~\ref{thm:noncrossingslekappa}}

The proof of the theorem will again just check the conditions from Lawler's characterisation of the $\sle$ measures. Among those, it is the restriction property that needs checking. To do this, it is in fact sufficient to show that
\begin{equation}\label{eq:rnderivktocompute}
\frac{\di \nu_{D \setminus K}(x,y)}{\di \nu_D(x,y)}(\eta) = \ind{\eta \subset D} \exp(\frac{\cc}{2} m_D(\eta, K))
\end{equation}
for $K \subset \overline{D}$ a hull attached to $\mcalB_r$, i.e. such that $K = \overline{K \cap D}$, $\overline{K} \cap \mcalB_l=\emptyset$, $D \setminus K$  and $K \cup \partial \D$ connected.
Indeed, this would characterise $\nu_D$ when seen as a law on connected hulls containing $(yx)$ (the left filling of its samples): since the $\slek$ measures should have the same Radon-Nikodym derivative, they must correspond to the same law on connected hulls containing $(yx)$, and the fact that they are simple curve---i.e. they coincide with their right boundary---is enough to assert that $\nu_D$ and $\slek$ also coincide as measures on simple curves from $x$ to $y$.

\begin{proof}
    Let $D':= D \setminus K$ be as above. Up to relabeling the boundary components $(\mcalB_j)_{1 \leq j \leq n}$ of $D$, we may assume that $(\mcalB_j)_{1 \leq j \leq k}$ are the inner boundary components of $D'$. We have to verify that $\nu_D$ satisfies Equation~\eqref{eq:rnderivktocompute}, i.e. that

    \begin{equation}\label{eq:rnderivktocompute2}
        \frac{\di \nu_{D'}}{\di \nu_{D}}(\eta) = \sum_{\substack{\bfb = (\bfb_k,1, \dots, 1) \\ \bfb_k \in \inset{\pm 1}^{n-k}}} \ind{\eta \in \mcalA_\bfb} \frac{\di \mbbP_{D'}^{\bfb_k}}{\di \mbbP_{D}^{\bfb}}(\eta) \frac{Z_{D',\bfb_k}}{Z_{D,\bfb}}
        = \sum_{\substack{\bfb = (\bfb_k,1, \dots, 1) \\ \bfb_k \in \inset{\pm 1}^{n-k}}} \ind{\eta \in \mcalA_{\bfb}, \eta \subset D'} \exp(\frac{\cc}{2}m_{D}(\eta, D \setminus D')).
    \end{equation}
    
    Given that on $\ind{\eta \in \mcalA_\bfb}$, $\mbbP_D^\bfb$ (resp. $\mbbP_{D'}^{\bfb_k}$) is a deterministic function of a realisation of $\cle_{\kappa}^D \otimes R_{h}^{D,\mcalB_\bfb^-}$ (resp. of $\cle_{\kappa}^{D'} \otimes R_{h}^{D',\mcalB_\bfb^-}$), it suffices to compute\footnote{We are abusing notation: the equation should be understood as
    \[
        \frac{ \di \cle_{\kappa}^{D'} \otimes R_{h}^{D',\mcalB_{\bfb}^-}}{\di \cle_{\kappa}^D \otimes R_{h}^{D,\mcalB_\bfb^-}}(F(\Gamma,K)) = \frac{ \di \cle_{\kappa}^{D'}}{\di \cle_{\kappa}^D}(F(\cdot,K))\restr{\cdot = \Gamma} \frac{ \di R_{h}^{D',\mcalB_{\bfb}^-}}{\di R_{h}^{D,\mcalB_\bfb^-}}(F(\Gamma,\cdot))\restr{\cdot = K}
    \]
    }
    \[
        \frac{\di \cle_{\kappa}^{D'} \otimes R_{h}^{D',\mcalB_\bfb^-}}{\di \cle_{\kappa}^{D} \otimes R_{h}^{D,\mcalB_{\bfb}^-}}(\eta) = \frac{\di \cle_{\kappa}^{D'}}{\di \cle_{\kappa}^{D}}(\eta) \frac{\di R_{h}^{D',\mcalB_\bfb^-}}{\di R_{h}^{D,\mcalB_{\bfb}^-}}(\eta).
    \]
    On the one hand, the sample of the Brownian loop soup in $D$ of intensity $\cc/2$ from which $\cle_\kappa^D$ is constructed consists of loops in $D'$ and loops intersecting $D'$. The former corresponds to a sample of a Brownian loop soup in $D'$ and the outermost boundaries of its clusters yield a realisation of $\cle_\kappa^{D'}$. One can therefore construct a curve $\tilde \eta'$ from $\cle_\kappa^{D'}$, and then explore the boundary of the loops in the remainder of the loop soup in $D$ to obtain $\tilde \eta$. Both curves are then such that either $\tilde \eta' = \tilde \eta$, or $\tilde \eta \not \subset D'$: and the former happens if and only if the loop soup does not contain loops intersecting both $\tilde\eta$ and $D \setminus D'$. Since the number of these loops is a Poisson random variable with parameter $\cc \cdot m_D(\tilde \eta, D \setminus D')/2$, we obtain that
    \[
        \frac{\di \cle_{\kappa}^{D'}}{\di \cle_{\kappa}^{D}}(\eta) = \Pb{\not \exists \ga \in \cle_{\kappa}^D: \ga \cap \eta \neq \emptyset, \ga \cap (D \setminus D') \neq \emptyset}^{-1} = \exp(\frac{\cc}{2} m_D(\eta, D \setminus D')).
    \]
    On the other hand, $R_{h}^{D,\mcalB_\bfb^-}$ can be obtained by first sampling $R_{h}^{D',\mcalB_{\bfb}^-}$, then an independent Poisson point process of Brownian excursions intersecting $D \setminus D'$ of intensity $h=h_\kappa$. Again here, the curve $\tilde \eta'$ obtained from $R_{h}^{D',\mcalB_\bfb^-}$ and the curve $\tilde \eta$ obtained from $R_{h}^{D,\mcalB_\bfb^-}$  are such that $\tilde \eta = \tilde \eta'$ or $\tilde \eta \not \subset D'$, the former happening if and only if the realisation of the PPP does not contain excursions intersecting $\mcalB_r'$ (note that these excursions necessarily hit $\eta$), i.e. if and only if no excursion in the realisation of $R_{h}^{D,\mcalB_\bfb^-}$ hit $D \setminus D'$. Once again we find that
    \[
        \frac{\di R_{h}^{D,\mcalB_\bfb^-}}{\di R_{h}^{D',\mcalB_{\bfb}^-}}(\eta) = \Pb{\not \exists \ga \in R_{h}^{D,\mcalB_\bfb^-}: \ga \cap D \setminus D' \neq \emptyset}^{-1} = \exp(\frac{\pi h}{4} (\norm{g_{D', \bfb_k}}_{\nabla(D'), \reg}^2 - \norm{g_{D, \bfb}}_{\nabla(D), \reg}^2)),
    \]
    using Lemma~\ref{lem:excursionproba}.
    The right-hand term of this equality cancels out with $Z_{D', \bfb_k}/Z_{D,\bfb}$ and concludes the proof.
\end{proof}

\appendix

\section{Existence of level-lines for compactified GFF on crossing admissible domains}\label{app:existencelevelline}
Here we explain in more details the existence of level-lines for the compactified GFF, in the crossing case: the argument is almost identical to the proof of Lemma 15 in \cite{aru_bounded-type_2019}. We replicate it here for the sake of completeness.
\begin{proposition}
    Let $(D,x,y)$ be a crossing admissible domain, $\bfb \in \Lambda^{n+1}$ and $\Phi_{D}^{g_\bfb} = \Phi_{D}^0 + g_\bfb$. There exists a local set coupling $(\Phi_{D}^{g_\bfb}, \eta, h_\eta)$ where $h_\eta$ is multivalued in $D\setminus \eta$ with constant branches equal to $\lam + 2n\lam$ on the event $\mcalA_\bfb$. $\eta$ is almost surely a simple curve from $x$ to $y$ in $\ann{p}$, measurable w.r.t. $\Phi_{D}^{g_\bfb}$, and its winding at $y$ is almost surely equal to $\sbin + \al$, with $\al$ the representative of $(\arg y - \arg x)/(2 \pi)$ in $[0,1)$.
\end{proposition}
\begin{proof}
    By conformal invariance we can assume that $D$ is given by an annulus $\ann{p}$ minus a set $K$ corresponding to a finite union of closed disks whose boundaries are given by $\mcalB_j, 1 \leq j \leq n, n \geq 0$.
    Given $D' \subset D$ a test domain, one can define the level-line $\eta$ of $\Phi_D^{g_D}\coloneqq \Phi_D^0 + g_D$ up until its exit time of $D'$ as the level-line in $D'$ of a GFF $\Phi_{D'}^{g_D\restr{D'}} \coloneqq \Phi_{D'}^0 + g_D\restr{D'}$. Uniqueness and measurability of this level-line w.r.t the latter GFF, as proved in the seminal works \cite{dubedat_sle_2009, schramm_contour_2013}, imply that this construction is compatible between different test domains so that $\eta$ is well-defined and actually a measurable function of $\Phi_D^{g_D}$.
    Similarly, since all branches of $g_D$ are such that $g_D$ is constant on each connected component of $\partial D \setminus \inset{x,y}$ and equal to $(2k+1)\lambda$ for some $k \in \Z$, then almost surely $\eta$ doesn't hit $\partial D \setminus \inset{y}$: for any simply connected domain $D' \subset D$, any point $y'$ of $(\partial D \cap \partial D') \setminus \inset{x,y}$ admits a neighbourhood $J$ in $D'$ such that $g_\bfb \restr{J \cap \partial D'}$ is equal to $(2k+1)\lambda$ for some $k \in \Z$, so by \cite[Lemma 4.2]{powell_level_2017} the probability that the level-line of $\Phi_{D'}^{g_D\restr{D'}}$ converges to a point in $J \cap \partial D'$ before exiting $D'$ is zero. On the other hand, if $\partial D' \ni y$, then there is a neighbourhood $J$ of $y$ in $D'$ such that $g_D \restr{J \cap \partial D'}$ is piecewise constant equal to $(2k-1)\lam$ (resp. $(2k+1)\lam$) on the part of $J \cap \partial D'$ in the clockwise direction (resp. anticlockwise) from $y$, for some $k \in \Z$. This $k$ is equal to zero if and only if the winding of $D'$ from $x$ to $y$ is given by\footnote{In the sense that any curve $\ga$ from $x$ to $y$ in $D'$ has winding given by $\sbin + \al$.} $\sbin + \al$, so if it is not the case by \cite[Lemma 4.2]{powell_level_2017} again the probability that the level-line of $\Phi_{D'}^{g_D\restr{D'}}$ converges to $y$ before exiting $D'$ is zero. We conclude by the construction of the level-line in $D$ and the union bound (for instance, considering all $D'$ corresponding the interior of finite unions of dyadic squares such that $D'$ is simply connected and contains $x$) that $\eta$ almost surely ends at $y$ with winding given by $\sbin + \al$.
    %
\end{proof}
In \cite{izyurov_hadamards_2013}, this result is established in the specific case $D=\ann{p}$, and it is further shown that $\eta$ is an annulus $\sle_4$ in $\ann{p}$, i.e. a Loewner chain for the annulus Loewner equation \cite{zhan_stochastic_2004} with driving function corresponding to a semimartingale with an explicit drift.

\section{Proof of Proposition~\ref{prop:detloopsoup}}\label{app:dubgeneral}

\noindent We remind the proposition:

\begin{myprop}{\ref{prop:detloopsoup}}
    For $D$ a simply connected proper domain of $\C$ and $\mcalB_1, \mcalB_2$ disjoint simple crosscuts such that $\dist(\mcalB_1,\mcalB_2)>0$ and $\mcalB_1$ is smooth, the following identity holds:
    \[
        \exp(-m_D(\mcalB_1,\mcalB_2)) = \det_F(I_{\mcalB_1}-(G_{D \setminus \mcalB_2})_{\mcalB_1}^{\mcalB_1} (H_{\partial (D \setminus \mcalB_1)}-H_{\partial (D \setminus (\mcalB_1 \cup \mcalB_2))})_{\mcalB_1}^{\mcalB_1}).
    \]
\end{myprop}

\begin{proof}
The identity holds true when $\mcalB_2$ is smooth and we have to extend it when $\mcalB_2$ is merely continuous. Since smooth crosscuts are dense in the space of continuous crosscuts with the Hausdorff topology, it is enough to argue continuity in $\mcalB_2$ of both sides of the identity. More precisely, fix two distinct endpoints $x,y$ on $\partial D$ and consider $\mcalX$ the set of smooth crosscuts $\mcalB_2$ having these two endpoints and disjoint from $\mcalB_1$, equipped with the Hausdorff metric. On the one hand, the map $\mcalB_2 \mapsto m_D(\mcalB_1, \mcalB_2)$ defined on $\mcalX$ is continuous (for instance by dominated convergence). On the other hand, since $A \mapsto \det_{\mathrm{F}}(I-A)$ is continuous over trace-class operators (with trace-class topology, \cite[Theorem 3.4]{simonTraceIdealsTheir2010}) the result follows provided we can prove the following:

\begin{claim}
    The map $\mcalB_2 \mapsto (G_{D \setminus \mcalB_2})_{\mcalB_1}^{\mcalB_1} (H_{\partial (D \setminus \mcalB_1)}-H_{\partial (D \setminus (\mcalB_1 \cup \mcalB_2))})_{\mcalB_1}^{\mcalB_1}$ is continuous from $\mcalX$ to the set of trace-class operators on $L^2(\mcalB_1)$.
\end{claim}
\noindent Indeed, consider a sequence of crosscuts $(\eta_k)_{k \geq 1} \subset \mcalX$ converging to some $\eta_\infty \in \mcalX$. Denote
    \[
        A_k := (G_{D \setminus \eta_k})_{\mcalB_1}^{\mcalB_1} (H_{\partial (D \setminus \mcalB_1)}-H_{\partial (D \setminus (\mcalB_1 \cup \eta_k))})_{\mcalB_1}^{\mcalB_1}  \quad\quad \text{ for } k \in \N \cup \inset{\infty}.
    \]
    These operators have kernel
    \[
        (A_k f)(x) = \int_{\mcalB_1} \underbrace{\int_{\mcalB_1}(G_{D \setminus \eta_k})(x,y) (H_{\partial (D \setminus \mcalB_1)}-H_{\partial (D \setminus (\mcalB_1 \cup \eta_k))})(y,z) \ldy}_{=: A_k(x,z)} f(z) \ld{z},
    \]
    which we recall corresponds to the density at $z$ of Brownian paths in $D$ started at $x$ that hit the curve $\eta_k$, then $\mcalB_1$. Since these kernels are smooth, by classical theory of integral operators they are trace class with trace equal to the integral of their kernel on the diagonal (see for instance \cite[Chapter 30.5]{lax2014functional}). Since $A_k(x,z) \xrightarrow[k \to \infty]{} A(x,z)$ pointwise and that for $\eps >0$ and $K$ large enough so that all $\eta_k, k \geq K$ lie in an $\eps$-neighborhood of $\eta$, $A_k(x,z)$ is bounded by the density at $z$ of Brownian paths in $D$ started at $x$ that hit the $\eps$-neighborhood of $\eta$, by dominated convergence it holds that $\Tr(A_k) \xrightarrow[k \to \infty]{} \Tr(A)$.
    
    \end{proof}

\section{Proof of Lemma~\ref{lem:rnderivativedisintegrationlemma}}\label{app:prooflemma}

\noindent We remind the lemma:

\begin{mylem}{\ref{lem:rnderivativedisintegrationlemma}}
    Let $X, Y_1,Y_2$ be random variables on the same probability space, taking values in a Polish space $E$, with associated distributions $\mbbP_X, \mbbP_{Y_1}, \mbbP_{Y_2}$ respectively. Denote also by $\mbbP_{Y_1 | Y_2 = y_2}$ the conditional law of $Y_1$ on ${Y_2 = y_2}$ (similarly for $\mbbP_{Y_2 | Y_1 = x}$) and assume that $\mbbP_X \ll \mbbP_{Y_1 | Y_2 = y_2}$, for $\mbbP_{Y_2}$-almost every $y_2 \in E$.
    Then the following equality between Radon-Nikodym derivatives holds $\mbbP_{X}$-a.s.
    \[
        \frac{\di \mbbP_X}{\di \mbbP_{Y_1}} (x) = \Eb{Y_2 | Y_1 = x}{\frac{\di \mbbP_X}{\di \mbbP_{Y_1 | Y_2 = y}}(x) \bigg|_{y=Y_2}}.
    \]
\end{mylem}

\begin{proof}
Let $f \in C_b(E)$. Denoting $R_{Y_2,y}(x)= \di \mbbP_X/ \di \mbbP_{Y_1 | Y_2 = y} (x)$, we compute:
\[
    \int_E f(x) \di \mbbP_X(x) = \int_E \int_E f(x) \di \mbbP_X(x) \di \mbbP_{Y_2}(y)
    = \int_E \int_E f(x) R_{Y_2,y}(x) \di \mbbP_{Y_1 | Y_2 = y}(x) \di \mbbP_{Y_2}(y).
\]
Using Bayes' formula, we rewrite the last term as
\[
\int_E \int_E f(x) R_{Y_2,y}(x) \di \mbbP_{Y_1 | Y_2 = y}(x) \di \mbbP_{Y_2}(y)
= \int_E \int_E f(x) R_{Y_2,y}(x) \di \mbbP_{Y_2 | Y_1 = x}(y) \di \mbbP_{Y_1}(x).
\]
But on the other hand, denoting $R(x) = \di \mbbP_X/\di \mbbP_{Y_1} (x)$, we have that
\[
\int_E f(x) \di \mbbP_X(x) = \int_E f(x) R(x) \di \mbbP_{Y_1} (x),
\]
from which we deduce that
\[
R(x) = \int_E R_{Y_2,y}(x) \di \mbbP_{Y_2 | Y_1 = x}(y). \qedhere
\]
\end{proof}

\nocite{*}
\bibliography{bibliography.bib}
\bibliographystyle{alphaurl}

\end{document}